\documentclass[11pt]{amsart}
\usepackage{amsmath}
\usepackage{amssymb}
\usepackage{amscd}

\def\NZQ{\mathbb}               
\def\NN{{\NZQ N}}

\def\ZZ{{\NZQ Z}}

\def\CC{{\NZQ C}}

\newtheorem{Theorem}{Theorem}[section]
\newtheorem{Lemma}[Theorem]{Lemma}
\newtheorem{Corollary}[Theorem]{Corollary}
\newtheorem{Proposition}[Theorem]{Proposition}

\newtheorem{Example}[Theorem]{Example}

\newtheorem{Definition}[Theorem]{Definition}

\let\epsilon\varepsilon
\let\phi=\varphi
\let\kappa=\varkappa

\begin{document}

\title{\'Etoiles and valuations}
\author{Steven Dale Cutkosky}
\thanks{partially supported by NSF}


\address{Steven Dale Cutkosky, Department of Mathematics,
University of Missouri, Columbia, MO 65211, USA}
\email{cutkoskys@missouri.edu}

\begin{abstract} We establish some properties of \'etoiles and associated valuations over complex analytic spaces, establishing that Abhyankar's inequality holds. We give some examples of pathological behavior of these valuations. We prove a regularization theorem for complex analytic morphisms. The property of a morphism being regular and the regularization of a morphism play a major role in this theory. 
\end{abstract}

\maketitle
\section{Introduction}

 A local blow up of an analytic space $X$ is a blow up $\pi:X'\rightarrow U$ where $U$ is an open subset of $X$ (in the Euclidean topology) and $\pi$ is the blow up of a closed analytic subspace of $U$.
(An inclusion of an open subset $U$ of $X$ into $X$ is a special case.)

Hironaka defined in \cite{H} and \cite{H3}  an \'etoile $e$ over an analytic space $X$ as a subcategory of the category of finite sequences of local blowups over $X$ which satisfies certain good properties.
 In particular, to each $\pi:X'\rightarrow X\in e$ there is an associated point $e_{X'} \in X'$, and 
 given a factorization 
$$
X'=X_n\stackrel{\pi_n}{\rightarrow}X_{n-1}\rightarrow \cdots \rightarrow X_1\stackrel{\pi_1}{\rightarrow }X
$$
by local blow ups, we have $\pi_i(e_{X_i})=e_{X_{i-1}}$ for all $i$.

In the situation of algebraic geometry (the category of algebraic blowups of an algebraic variety $X$ over a field $k$) an \'etoile $e$ can be represented by a valuation of the function field $k(X)$ of $X$ which dominates the local ring $\mathcal O_{X,e_X}$ (whose quotient field is $k(X)$). This is the original approach of Zariski \cite{LU}. 

The notion of an \'etoile $e$ on a complex analytic space $X$ cannot be immediately modeled in valuation theory, even when $X$ is irreducible and nonsingular, as there exist $\pi:X'\rightarrow X\in e$ such that $X'$ is not locally irreducible, and even when $X'$ is locally irreducible, $\mathcal O_{X',e_{X'}}$ is generally a very big extension field of $\mathcal O_{X,e_X}$.

Valuation theory is an important tool in the birational geometry of algebraic varieties, and it is useful to know which parts of the classical theory for algebraic function fields extend to \'etoiles on an irreducible nonsingular complex analytic space.

In Section \ref{Secval} we  associate to an \'etoile over a reduced complex analytic space  $X$ a valuation $\nu=\nu_e$ on a giant field which depends on the \'etoile $e$.
The valuation ring $V_e$ is constructed by taking the union of $\mathcal O_{X',e_{X'}}$ where $X'\rightarrow X\in e$ is a sequence of local blow ups from nonsingular  varieties.
We establish in Section \ref{Secval} that we have, as in the classical case of valuations of algebraic function fields, that
$$
\mbox{rank }\nu\le \mbox{ratrank }\nu\le \dim X,
$$
and if the rational rank $\mbox{ratrank }\nu=\dim X$, then the value group $\Gamma_{\nu}$ of $\nu$ is isomorphic (as an unordered group) to $\ZZ^{\dim X}$. 

The residue field of the valuation ring $V_e$ associated to an \'etoile $e$ is always isomorphic to  $\CC$. Thus $\nu$ is always a zero dimensional valuation  and we see that Abhyankar's inequality \cite{Ab} for a valuation of a field $K$ which dominates a Noetherian local ring whose quotient field is $K$, holds for the valuation $\nu_e$ associated to an \'etoile $e$.

Unlike in the case of algebraic function fields, a composite valuation which arises from an \'etoile can be very badly behaved, as is shown in the following example.
The existence of examples of this type was a major obstruction to a proof of local monomialization of analytic morphisms.

\begin{Example}(Example \ref{pathval}) There exists an \'etoile $e$ on $Y_0=\CC^4$ such that the valuation ring $V_e$ has a proper prime ideal 
$Q$
 such that there exists an infinite chain
of local blow ups (of a point in $Y_m$ if $m$ is even and of a nonsingular surface if $m$ is odd)
$$
 \cdots\rightarrow Y_m\rightarrow \cdots \rightarrow Y_1\rightarrow Y_0
$$
with $Y_m\rightarrow Y_0\in e$ for all $m$, such that the center of $Q$ on $Y_m$ has dimension 3 if $m$ is even and the center of $Q$ on $Y_m$ has dimension 2 if $m$ is odd.
\end{Example}

The construction begins with an example by
Hironaka, Lejeune and Teissier \cite{HLT}  of a germ of an  analytic map $\phi:(S,a)\rightarrow (V,b)$ from a surface to a 3-fold such that no functions in $\mathcal O_{V,b}$ vanish on the image of $\phi$ but the image becomes a two dimensional  analytic sub variety (a surface) after blowing up $b$.

Hironaka (\cite{H} and \cite{H3}) defines 
La Vo\^ute \'Etoil\'ee as
$$
\mathcal E_X=\mbox{ set of all \'etoiles over $X$ with a topology making $P_X:\mathcal E_X\rightarrow X$, $e\mapsto e_X$ continuous.}
$$
Hironaka proves that  $p_X$ is proper.
This theorem is a generalization of Zariski's theorem \cite{Z}  on the quasi compactness of the Zariski Riemann manifold of an algebraic function field.

If $\phi:Y\rightarrow X$ is a dominant morphism of algebraic varieties over a field $k$ (the Zariski closure of $\phi(Y)$ in $X$ is equal to $X$) then we have a natural inclusion of algebraic function fields $k(X)\rightarrow k(Y)$. Thus a valuation of $k(Y)$ restricts to a valuation of $k(X)$ and a valuation of $k(X)$ can be extended to a valuation of $k(Y)$ (Chapter VI \cite{ZS}).

However, the situation is much more subtle in the case of  complex analytic morphisms of complex analytic spaces (as is exploited in the construction of the above example).

The most useful generalization of the notion of a dominant morphism of algebraic varieties to  analytic morphisms of complex analytic spaces is a regular morphism.

Let $\phi:(Y,b)\rightarrow (X,a)$ be a germ of a morphism of  complex analytic spaces. If $X$ and $Y$ are varieties, then $\phi$ is regular if $\phi(Y)$ contains an open subset of $X$ (in the Euclidean topology).
Let $\mbox{reg}(Y)$ be the nonsingular locus of $Y$.
The morphism $\phi$ is regular if and only if
the open set 
$$
U=\{p\in \mbox{reg}(Y)\mid \mbox{rank }d\phi_p=\dim X\}
$$
is nonempty (see Section \ref{Secpre} or \cite{BM2}).

Gabri\`elov gave an example in \cite{Gab2} showing that  if $\phi$ is not regular, it is possible for the map $\mathcal O_{X,a}\rightarrow \mathcal O_{Y,b}$  of analytic local rings to be injective, but the induced map on completions 
 $\hat{\mathcal O}_{X,a}\rightarrow \hat{\mathcal O}_{Y,b}$ to be not injective. (The Zariski subspace theorem (10.6) \cite{Ab3} fails for analytic maps). 
 Gabri\`elov's example begins with an earlier example of Osgood (explained in Example 2.3 \cite{BM2}).
 If $\phi$ is regular, Gabri\`elov \cite{Gab}  showed that the Zariski subspace theorem holds for regular morphisms.

Suppose that $\phi:X\rightarrow Y$ is a morphism of reduced, irreducible, locally irreducible  complex analytic spaces. It is shown in Proposition \ref{Prop2} that if $e$ is an \'etoile on $Y$, then $e$ induces an \'etoile on $X$ if and only if $\phi$ is regular. It is shown in Proposition \ref{Etoileext} that if $f$ is an \'etoile on $X$, then there exists an \'etoile $e$ on $Y$ that induces $f$ if and only if $\phi$ is regular.

We prove the following regularization theorem,  in Theorem \ref{TheoremRegFin} of Section \ref{Secreg}.

\begin{Theorem}(Theorem \ref{TheoremRegFin}) Suppose that $\phi:Y\rightarrow X$ is a morphism of reduced  complex analytic spaces and $e\in \mathcal E_Y$
is an \'etoile over $Y$. Then there exists a commutative diagram of morphisms
$$
\begin{array}{ccc}
Y_e&\stackrel{\phi_e}{\rightarrow}&X_e\\
\delta\downarrow &&\downarrow \gamma\\
Y&\stackrel{\phi}{\rightarrow}& X
\end{array}
$$
such that $\delta\in e$ is a finite product of local blow ups of nonsingular analytic sub varieties, $\gamma$ is a finite product of local blow ups of nonsingular analytic sub varieties, $Y_e$ and $X_e$ are smooth analytic spaces and $ \phi_e$ is a regular analytic morphism to a nonsingular analytic sub variety of $X_e$.
\end{Theorem}

The proof is rather  delicate, and requires the analysis of \'etoiles of Section \ref{SecBup}. An essential ingredient is the local flattening theorem \cite{HLT} of Hironaka, Lejeune and Teissier   or the later proof by Hironaka \cite{H3}. The local flattening theorem is with reference to a fixed \'etoile $f$ on $X$. We show in Proposition \ref{Prop2} that an \'etoile $e$ on $Y$ only induces an \'etoile on $X$ if $\phi$ is regular. Because of this observation, Theorem \ref{TheoremRegFin} does not follow directly from the local flattening theorem.

We state below the principal theorem for local monomialization of complex analytic spaces along an \'etoile in \cite{C4}. The first proof of Theorem \ref{TheoremB} required Theorem \ref{TheoremRegFin}, but the final proof does not require this, and in fact gives an alternate proof of Theorem \ref{TheoremRegFin} which does not require the local flattening theorem.

\begin{Definition} Suppose that $\phi:Y\rightarrow X$ is an analytic  morphism of complex  analytic manifolds and $p\in Y$. We will say that $\phi$ is monomial at $p$ if there exist regular parameters $x_1,\ldots,x_m$ in $\mathcal O_{X,\phi(p)}$ and $y_1,\ldots,y_n$ in $\mathcal O_{Y,p}$, $r\le m$ and $c_{ij}\in \NN$ such that 
$$
\phi^*(x_i)=\prod_{j=1}^ny_j^{c_{ij}}\mbox{ for }1\le i\le r
$$
with $\phi^*(x_i)=0$ for $r<i\le m$ and 
$\mbox{rank}(c_{ij})=m$.  
\vskip .2truein
We will say that $y_1y_2\cdots y_n=0$ is a local toroidal structure $O$ at $p$
\end{Definition}

\begin{Theorem}\label{TheoremB}(\cite{C4}) Suppose that $\phi:Y\rightarrow X$ is a morphism of reduced complex analytic spaces, $A$ is a closed  analytic  subspace of $Y$  and $e\in\mathcal E_Y$ is an \'etoile over $Y$. Then there exists a commutative diagram of complex analytic morphisms
$$
\begin{array}{ccc}
Y_{e}&\stackrel{\phi_{e}}{\rightarrow}&X_e\\
\beta\downarrow &&\downarrow \alpha\\
Y&\stackrel{\phi}{\rightarrow}& X
\end{array}
$$
such that $\beta\in e$ is a finite product of local blow ups of nonsingular analytic sub varieties, $\alpha$ is a finite product of local blow ups of nonsingular analytic sub varieties, $Y_e$ and $ X_e$ are nonsingular analytic  spaces and $\phi_e$ is a  monomial  analytic morphism for a toroidal structure $O_e$ on $Y_e$.
Further, either the preimage  of $A$ in $Y_e$ is equal to $Y_e$, or $\mathcal I_A\mathcal O_{Y_e}=\mathcal O_{Y_e}(-G)$ where $\mathcal I_A$ is the ideal sheaf in $\mathcal O_Y$ of the analytic subspace $A$  of $Y$, $G$ is an effective divisor which is supported on $O_e$, and has the further condition that the restriction $(Y_e\setminus O_e) \rightarrow Y$ is an open embedding.  
\end{Theorem}

Local monomialization theorems for real analytic morphisms  are also proven in \cite{C4}. Local monomialization along an arbitrary valuation is proven for morphisms of algebraic varieties in characteristic zero in \cite{C1} and \cite{C2}. Counterexamples to local monomialization for a morphism of characteristic $p>0$ algebraic varieties is given in \cite{C5}. A couple of  interesting recent papers which address local monomialization of analytic morphisms and applications are 
\cite{De} by Jan Denef and \cite{Li} by Ben Lichtin.

Hironaka used the local flattening theorem and the fiber cutting lemma in \cite{H3} to prove rectilinearization of real sub analytic sets. A couple of more recent proofs of rectilinearization are given in \cite{DD} and \cite{BM3}. We deduce rectilinearization in \cite{C6} from our local monomialization theorem \cite{C4}, without using local flattening or the fiber cutting lemma.

We thank Jan Denef for suggesting the local monomialization problem for analytic morphisms, and for discussion, encouragement and explanation of possible applications. We also thank Bernard Teissier for discussions on this and related problems.

\section{Preliminaries on  complex analytic spaces}\label{Secpre}

In this section we recall some basic properties of analytic local rings and complex analytic spaces.

\begin{Proposition}\label{anring} Suppose that $X$ is a complex analytic space and $p\in X$. Then 
\begin{enumerate}
\item[1.] $\mathcal O_{X,p}$ is a Noetherian, Henselian, excellent local ring.
\item[2.] $\mathcal O_{X,p}$ is equidimensional if and only if its completion $\hat{\mathcal O}_{X,p}$ is equidimensional.
\item[3.] $\mathcal O_{X,p}$ is reduced if and only if $\hat{\mathcal O}_{X,p}$ is reduced.
\item[4.] $\mathcal O_{X,p}$ is a domain if and only if $\hat{\mathcal O}_{X,p}$ is a domain.
\end{enumerate}
\end{Proposition}

\begin{proof} The fact that $\mathcal O_{X,p}$ is Noetherian and Henselian  is proven in Theorem 45.5, and by fact 43.4, \cite{N}. Excellence is proven   in Section 18 of \cite{EGAIV} (or Theorem 102, page 291 \cite{Ma}), and by (ii) of Scholie 7.8.3 \cite{EGAIV}.
Let $A=\mathcal O_{X,p}$. Since $A$ is a local ring, the natural map $A\rightarrow \hat A$ is
an inclusion. $A$ and $\hat A$ have the same Krull dimension (formula 1' of page 175 \cite{Ma}). Statements 2 and 3 follow from (vii) and (x) of Scholie 7.8.3 \cite{EGAIV}.
 Further,   $\hat{\mathcal O}_{X,p}$ is a domain
  if and only if $\mathcal O_{X,p}$ is a domain by Corollary 18.9.2 \cite{EGAIV}. 
\end{proof}

The dimension $\dim E$ of a subset $E$ of a complex analytic space $X$ and the local dimension $\dim_a E$ of $E$  at a point $a\in X$ are defined
in II.1 and V.4.4 of Lojasiewicz's excellent book \cite{L}. If $E$ is an analytic space, then   $\dim_aE$ is the Krull dimension of $\mathcal O_{E,a}$.

\begin{Lemma} Suppose that $Y$ is a reduced complex analytic space, and $\pi:B\rightarrow Y$ is the blow up of a closed complex analytic
 subspace $E$ of $Y$.  Then $Y$ is reduced. If $Y$ is equidimensional, then $B$ is equidimensional. 
\end{Lemma}

\begin{proof} Suppose that $q\in B$. Let $p=\pi(q)\in Y$. Then $A=\mathcal O_{Y,p}$ is reduced (respectively equidimensional if $Y$ is equidimensional) 
Let $I=\mathcal I_{E,p}$ be the stalk of the ideal of $E$ in $\mathcal O_{Y,p}$. The $A$-scheme $P=\mbox{Proj}(\bigoplus_{n\ge 0} I^n)$ is
reduced (respectively equidimensional if $Y$ is equidimensional) (Section 7 of Chapter II \cite{Ha}). There exists a point $q'$ above $p$ in $P$ such that $\mathcal O_{B,q}$ is the analytification of the local ring $\mathcal O_{P,q'}$, so that these two local rings have the same completion. $\mathcal O_{P,q'}$ is excellent (by (ii) of Scholie 7.8.3 \cite{EGAIV}). Thus 
the completion of $\mathcal O_{P,q'}$ is
reduced  (respectively equidimensional if $Y$ is equidimensional), and so $\mathcal O_{B,q}$ is reduced (respectively equidimensional if $Y$ is equidimensional).
\end{proof}

Let $\mbox{Reg}(X)$ denote the open subset of nonsingular points of a complex analytic space $X$.

Suppose that $\phi:X\rightarrow Y$ is a morphism of complex analytic spaces. Suppose that $p\in X$ and $q=\phi(p)$. Let 
$\phi_p^*:\mathcal O_{Y,q}\rightarrow \mathcal O_{X,p}$ be the induced homomorphism of germs of analytic functions, with associated homomorphism
${\hat \phi}_p^*:\hat {\mathcal O}_{Y,q}
\rightarrow \hat {\mathcal O}_{X,p}$ of complete local rings. 

Suppose that $X$  and   $Y$ are reduced and $\phi:X\rightarrow Y$ is a morphism. For $a\in \mbox{Reg}(X)$, define
$\mbox{rank}_a(\phi)$ to be the rank of the map on tangent spaces $d\phi_a:T(X)_{a}\rightarrow T(Y)_{\phi(a)}$, and
$$
\mbox{rank}(\phi)=\max\{\mbox{rank}_a\phi\mid a\in \mbox{reg}(X)\},
$$
and for $p\in X$ (possibly not in $\mbox{Reg}(X)$),
$$
\mbox{rank}_p(\phi)=\min\{\mbox{rank}(\phi|U)\mbox{ such that } U\mbox{ is an open neighborhood of $p$}\}.
$$
We have that $\mbox{rank}(\phi)=\dim\phi(X)$ by Theorem 4 of V.3.3 \cite{L}.

If $X$ is irreducible, then
for $p$ in $X$,
we have that 
\begin{equation}\label{eqreg1}
\mbox{rank}(\phi)=\mbox{rank}_p(\phi)=\dim_{\phi(p)} \phi(X)=\dim \phi(X),
\end{equation}
by Theorem 4 and Corollary 2 of V.3.3 \cite{L}.

\begin{Definition} Suppose that $\phi:X\rightarrow Y$ is a morphism of reduced, irreducible   complex analytic spaces. 
$\phi$ is said to be regular  if $\phi(X)$ contains an open subset of $Y$.
\end{Definition}

Suppose that $\phi:X\rightarrow Y$ is an analytic morphism of connected complex analytic manifolds. Let
$$
U=\{q\in X\mid \mbox{rank }d\phi_q=\dim Y\}.
$$
Then $Z=X\setminus U$ is an analytic subspace of  $X$ and by Theorem 4 of V.3.3 \cite{L}, $U\ne \emptyset$ if and only if $\phi$ is regular.

\begin{Lemma}\label{Lemma5} Suppose that $\phi:X\rightarrow Y$ is a  regular morphism of reduced, irreducible, locally irreducible complex analytic spaces.
 Then there exists a nowhere dense closed analytic subset $G$ of $X$ such that $\phi(X\setminus G)$ is an open subset of $Y$, the restriction $\phi|(X\setminus G)$ is 
 an open mapping and $\dim \phi(G)<\dim Y$.
 \end{Lemma}

 \begin{proof} For $x\in X$, let $\ell_x\phi$ be the germ at $x$ of the fiber of $\phi(x)$ by $\phi$ (defined on page 267 of V.3.2 \cite{L}).
  By the Cartan Remmert Theorem (Theorem 5, V.3.3 \cite{L}), $\dim \ell_x\phi$ is upper semi continuous
on $X$ in the analytic Zariski topology. Let 
$$
t=\min\{\dim \ell_x\phi\mid x\in X\}.
$$
We have that 
$$
t=\dim X- \mbox{rank}(\phi)=\dim X - \dim \phi(X)=\dim X - \dim Y,
$$
by formula (1) of V.3.3 \cite{L}, Theorem 4, V.3.3 \cite{L} and the assumption that $\phi$ is regular.
 Now
$$
G=\{x\in X\mid \dim \ell_x\phi>t\}
$$
 is a proper subset of $X$ which is closed in the analytic Zariski topology,
so that it is a thin set (Proposition of II.3.5 \cite{L}),
and $V=X\setminus G$ is an open subset of $X$  on which $\phi$ has constant minimal fiber dimension $t$.
Further, by  Remmert's Rank Theorem (Theorem 1 of V.6 \cite{L}), for every $p\in X$ there exist arbitrarily small neighborhoods $U$ of $p$ in $X$ such that
$\phi(U)$ is locally analytic in $Y$,  of  dimension $\dim X -t$. We further have that
$$
\dim \phi(G)\le \dim G-(t+1)<\dim X-t=\dim Y
$$
 by Theorem 2, V.3.2 \cite{L}, since $\dim \ell_x(\phi\mid G)>t$ for all $x\in G$.

Finally, by  Remmert's Open Mapping Theorem (Theorem 2, V.6, \cite{L}),  the restriction of $\phi$
to $X\setminus G$ is an open mapping to $Y$, since $t=\dim X -\dim Y$.
\end{proof}

\begin{Proposition}\label{Propreg} Suppose that $\phi:X\rightarrow Y$ is a morphism of  irreducible nonsingular complex analytic spaces, and 
$\phi$ is regular. Then $\hat\phi^*:\hat{\mathcal O}_{Y,\phi(p)}\rightarrow \hat{\mathcal O}_{X,p}$ is 1-1 for all $p\in X$.
\end{Proposition}

\begin{proof} We have that for all $p\in X$, $\mbox{rank}_p(\phi)=\mbox{rank}(\phi)= \dim \phi(X)$, by  (\ref{eqreg1}), so that 
$\hat\phi_p:\hat{\mathcal O}_{Y,\phi(p)}\rightarrow \hat {\mathcal O}_{X,p}$ is 1-1 for all $p\in X$ by Lemma 4.2 \cite{Gab}.
\end{proof} 

\begin{Lemma}\label{coeffield} Suppose that $A$ is an analytic local ring and $\mathfrak p$ is a prime ideal in $A$. Then there exists a field $K\subset A_{\mathfrak p}$ such that the induced map to the residue field $K\rightarrow (A/\mathfrak p)_{\mathfrak p}$ is a finite field extension.
\end{Lemma}

\begin{proof} We have a representation $A\cong \mathcal O_n/I$ for some $n$ where $I$ is an ideal in the ring $\mathcal O_n$  of germs of analytic functions at the origin in $\CC^n$. There is a prime ideal $P$ in $\mathcal O_n$ containing $I$, such that $P/I\cong\mathfrak p$.
By the Proposition of III.2.5 \cite{L}, there exists a set of coordinates $z_1,\ldots, z_n$ in $\mathcal O_n$, so that $\mathcal O_n=C\{z_1,\ldots,z_n\}$,
and $k\le n$ such that the induced map $\mathcal \CC\{z_1,\ldots,z_k\}\rightarrow \mathcal O_n/P$ is a 1-1 finite map. In particular,
$\CC\{z_1,\ldots,z_k\}\cap P=(0)$. Thus the induced map $\CC\{z_1,\ldots,z_k\}\rightarrow \mathcal O_n/I\cong A$ is 1-1 and $\mathfrak p\cap \CC\{z_1,\ldots,z_k\}=(0)$, so that we have an inclusion of the quotient field $K=\CC\{\{z_1,\ldots,z_k\}\}$ into $A_{\mathfrak p}$, such that 
$(A/\mathfrak p){\mathfrak p}$ is finite over $K$.
\end{proof}

A fundamental theorem  in complex analytic geometry is Hironaka's theorem \cite{H2} on the existence of a resolution of singularities of a reduced complex space $X$ (which is countable at infinity), by a sequence of blow ups of nonsingular subvarieties. The sequence is finite if $X$ is compact. In the case of a germ $(X,p)$, this already follows from  Hironaka's Theorem $I_2^{N,n}$ \cite{H1}, since $\mathcal O_{X,p}$ is excellent and reduced. The general Theorem
is proven in the monograph of Aroca, Hironaka and Vicente \cite{AHV}. A simplified proof is given in \cite{BM}.

\section{La Vo\^ute \'Etoil\'ee}

In this section, we recall some definitions and results from \cite{H}.

\begin{Definition}(Definition 1.4 \cite{H}) A morphism $\pi:Y'\rightarrow Y$ of complex analytic spaces is called strict if there exists a complex
analytic subspace $E'$ of $Y'$ such that $\pi$ is \'etale at all points of $Y'\setminus E'$ and $(Y',E')$ is minimal, in the sense that if $Z$ is a closed analytic subspace of $Y'$ such that $Y'\setminus E'=Z\setminus E'$, then $Y'=Z$.
\end{Definition}

Let $Y$ be a complex analytic space. A local blow up of $Y$ (page 418 \cite{H}) is the morphism $\pi:Y'\rightarrow Y$ determined by given $(U,E,\pi)$ where
$U$ is an open subset of $Y$, $E$ is a  closed  analytic subspace of $U$ and $\pi$ is the composite of the blow up of $E$ with the inclusion of $U$ into $Y$.

A sequence of local blow ups of $Y$ is the composite of a finite sequence of local blow ups $(U_i,E_i,\pi_i)$. Any sequence of local blow ups is strict (\cite{H}).

Let $Y$ be a complex analytic space. $\mathcal E(Y)$ will denote the category of morphisms $\pi:Y'\rightarrow Y$ which are sequence of local blow ups.
For $\pi_1:Y_1\rightarrow Y\in \mathcal E(Y)$ and $\pi_2: Y_2\rightarrow Y\in \mathcal E(Y)$, ${\rm Hom}(\pi_1,\pi_2)$ denotes the $Y$-morphisms
$Y_1\rightarrow Y_2$ (morphisms which factor $\pi_1$ and $\pi_2$). ${\rm Hom}(\pi_1,\pi_2)$ has at most one element.

\begin{Definition}\label{EtoileDef} (Definition 2.1 \cite{H}) Let $Y$ be a complex analytic space. An \'etoile over $Y$ is a subcategory $e$ of $\mathcal E(Y)$ having the following properties:
\begin{enumerate}
\item[1)] If $\pi:Y'\rightarrow Y\in e$ then $Y'\ne \emptyset$.
\item[2)] If $\pi_i\in e$ for $i=1,2$, then there exists $\pi_3\in e$ which dominates $\pi_1$ and $\pi_2$; that is, ${\rm Hom}(\pi_3,\pi_i)\ne 0$
for $i=1,2$.
\item[3)] For all $\pi_1;Y_1\rightarrow Y\in e$, there exists $\pi_2:Y_2\rightarrow Y\in e$ such that there exists $q\in {\rm Hom}(\pi_2,\pi)$, and the image $q(Y_2)$ is relatively compact in $Y_1$.
\item[4)] (maximality) If $e'$ is a subcategory of $\mathcal E(Y)$ that contains $e$ and satisfies the above conditions 1) - 3), then $e'=e$.
\end{enumerate}
\end{Definition}

The set of all \'etoiles over $Y$ is denoted by $\mathcal E_Y$.

Using property 3), Hironaka shows that
for $e\in \mathcal E_Y$, and $\pi:Y'\rightarrow Y\in e$, there exists a uniquely determined point $p_{\pi}(e)\in Y'$ (which we will also denote by $e_{Y'}$) which has the property that
if $\alpha:Z\rightarrow Y\in e$ factors as
$$
Z\stackrel{\beta}{\rightarrow} Y' \stackrel{\pi}{\rightarrow} Y,
$$
then $\beta(p_{\alpha}(e))=p_{\pi}(e)$. In particular, we have a natural map $p_Y:\mathcal E_Y\rightarrow Y$ defined by $p_Y(e)=p_{\rm id}(e)$.
Hironaka shows (in Theorem 3.4 \cite{H}) that $\mathcal E_Y$ has  a natural topology so that $p_Y$ is continuous, surjective and proper.

$\mathcal E_Y$ with this topology is called ``La vo\^ute \'etoil\'ee.

\section{blow ups and morphisms along an \'etoile and the distinguished irreducible component} \label{SecBup}

The join  of $\pi_1,\pi_2\in \mathcal E(Y)$ is defined in Proposition 2.9 \cite{H}. We will denote this join by $J(\pi_1,\pi_2)$.
It is a morphism $J(\pi_1,\pi_2):Y_J\rightarrow Y$. 
It has the following universal property: Suppose that $f:Z\rightarrow Y$ is a strict morphism. Then there exists a $Y$-morphism 
$Z\rightarrow Y_J$ if and only if there exist $Y$-morphisms $Z\rightarrow Y_1$ and $Z\rightarrow Y_2$.
It follows from 2.9.2 \cite{H} that if $\pi_1,\pi_2,\in e\in \mathcal E_Y$, then $J(\pi_1,\pi_2)\in e$. 
We describe the construction of Proposition 2.9 \cite{H}. In the case when $\pi_1$ and $\pi_2$ are each local blowups,
which are described by the data $(U_i,E_i,\pi_i)$, $J(\pi_1,\pi_2)$ is the blow up 
$$
J(\pi_1,\pi_2):Y_J=B(\mathcal I_{E_1}\mathcal I_{E_2}\mathcal O_Y|U_1\cap U_2)\rightarrow Y.
$$
Now suppose that $\pi_1$ is a product $\alpha_0\alpha_{1}\cdots \alpha_r$ where $\alpha_i:Y_{i+1}\rightarrow Y_i$ are local blow ups defined by the data
$(U_i,E_i,\alpha_i)$, and $\pi_2$ is a product $\alpha_0'\alpha_1'\cdots\alpha_r'$ where $\alpha_i':Y_{i+1}'\rightarrow Y_i'$ are local blow ups defined by the data $(U_i',E_i',\alpha_i')$. We may assume (by composing with identity maps) that the length of each sequence is a common value $r$.
We define $J(\pi_1,\pi_2)$ by induction on $r$. Assume that $J_{r}=J(\alpha_{0}\alpha_1\cdots\alpha_{r-1},\alpha_{0}'\alpha_1'\cdots\alpha_{r-1}')$ has been constructed,
with projections $\gamma:Y_{J_r}\rightarrow Y_{r}$ and $\delta:Y_{J_r}\rightarrow Y_{r}'$. Then we define $J(\pi_1,\pi_2)$ to be the blow up
$$
J(\pi_1,\pi_2):Y_{J}=B(\mathcal I_{E_{r}}\mathcal I_{E'_{r}}\mathcal O_{J_{r}}|\gamma^{-1}(U_r)\cap \delta^{-1}(U_r'))\rightarrow Y.
$$

Suppose that $e\in \mathcal E_Y$ is an \'etoile. By Lemma 2.3 \cite{H}, there exists a point $p_{\pi}(e)\in Y'$ for all 
$\pi:Y'\rightarrow Y\in e$, such that if $\pi_1,\pi_2\in e$ and $\phi\in\mbox{Hom}(\pi_1,\pi_2)$, then
\begin{equation}\label{eq3}
p_{\pi_2}(e)=\phi(p_{\pi_1}(e)).
\end{equation}
(Condition 3) of Definition \ref{EtoileDef} is essential for this result.)
Suppose that $Y$ is a reduced complex analytic space, $e\in \mathcal E_Y$ and $\pi:Y'\rightarrow Y\in e$. Suppose that $U$ is a neighborhood of
$p_{\pi}(e)\in Y'$. We will define the distinguished  irreducible component $\mbox{DC}_e(U)$ of $U$. Let 
$F_1, F_2,\ldots, F_s$ be the distinct irreducible components of $U$. Let 
$\pi':U'\rightarrow U$ be a global
blowup of a nowhere dense closed algebraic set, which separates out the irreducible components of $U$ into distinct connected components $Z_1,\ldots, Z_s$ such that $\pi'(Z_i)\subset F_i$ for all $i$, and $Z_i\rightarrow Y_i$ is strict (such as a resolution of singularities of $U$). Then $\pi'\pi\in e$ by Corollary 2.11.4 \cite{H}. There exists a unique component $Z_i$ of $U'$ such that $p_{\pi\pi'}(e)\in Z_i$. Define $\mbox{DC}_e(U)=F_i$.
This  is well defined, since if $\pi'':U''\rightarrow U$ is another global blowup of a nowhere dense closed analytic subset of $U$ which separates the components of $U$, then by 2) of Definition \ref{EtoileDef},
 there exists $\lambda:W\rightarrow Y\in e$ and maps $\alpha\in\mbox{Hom}(\lambda,\pi')$, 
$\beta\in\mbox{Hom}(\lambda,\pi'')$ such that $\alpha(p_{\lambda}(e))=p_{\pi\pi'}(e)$ and $\beta(p_{\lambda}(e))=p_{\pi\pi''}(e)$. Since
$\pi'$ and $\pi''$ are blow ups of nowhere dense closed analytic sets, there is an open subset of $U'$ which intersects all components of $U'$ non trivially
which is isomorphic to an open subset of $U''$ which intersects all components of $U''$ non trivially. Thus the component of $U''$ which 
contains $p_{\pi''}(e)$ must map to $\mbox{DC}_e(U)$.

\begin{Lemma}\label{Lemma1} Suppose that $Y=Y_0$ is a reduced complex analytic space, $e\in \mathcal E_Y$ and $\pi:Y'\rightarrow Y\in e$. Suppose that
$\pi$ has a factorization $\pi=\pi_0\pi_{1}\cdots \pi_r$ where $\pi_i:Y_{i+1}\rightarrow Y_{i}$ are local blow ups determined by the data
$(U_i,E_i,\pi_i)$. Then $\pi_{0}\cdots \pi_i\in e$, $p_{\pi_0\cdots\pi_{i-1}}(e)\in U_{i}$ and $\mbox{DC}_e(U_i)\not\subset E_i$ for all $i$.
\end{Lemma}

\begin{proof} We will first show that $\pi_{0}\cdots \pi_i\in e$ for all $i$.  We will use the criterion of Lemma 2.10 on page 431 of \cite{H}.
Suppose that $\phi_{\alpha}:Z_{\alpha}\rightarrow Y\in e$. We must show that there exists $\phi_{\beta}:Z_{\beta}\rightarrow Y\in e$ such that 
$\mbox{Hom}(\phi_{\beta},\phi_{\alpha})\ne \emptyset$, and if $J(\phi_{\beta},\pi_0\cdots\pi_i):Z_J\rightarrow Y$ is the join, then the natural image of 
$Z_J$ in $Y_{i+1}$ is relatively compact and non empty.

By 2) and 3) of Definition \ref{EtoileDef}, there exists $\phi_{\beta}:Z_{\beta}\rightarrow Y\in e$ such that
$\mbox{Hom}(Z_{\beta},Z_{\alpha})\ne \emptyset$, $\mbox{Hom}(Z_{\beta},Y')\ne \emptyset$ and if $q:Z_{\beta}\rightarrow Y'$ is the induced map,
then $q(Z_{\beta})$ is relatively compact in $Y'$. Let $J(\phi_{\beta},\pi_0\cdots\pi_i):Z_J\rightarrow Y$ be the join.

Then $Z_J=Z_{\beta}$ since $\pi_{\beta}$ factors through $\pi_0\cdots \pi_i$.
Since the image of $Z_{\beta}$ is relatively compact in $Y'$, the image of $Z_{\beta}$ in $Y_{i+1}$ is also relatively compact. 
The fact that $p_{\pi_0\cdots \pi_i}(e)\in U_{i+1}$ for all $i$ now follows from (\ref{eq3}).

Let $h=\pi_{0}\cdots \pi_{i-1}$. Let $\lambda:Z\rightarrow U_i$ be a global blow up which separates the irreducible components of $U_i$. Then 
$h\lambda\in e$. Since $h\pi_i\in e$, there exists (by 2) of Definition \ref{EtoileDef}) $\tau:W\rightarrow Y\in e$ with factorizations 
$$
\begin{array}{ccc}
&W&\\
\alpha\swarrow&&\searrow\beta\\
Y_{i+1}&&Z\\
\pi_i\searrow&&\swarrow\lambda\\
&U_i&\\
&\downarrow h&\\
&Y&
\end{array}
$$
Let $H$ be the irreducible component of $W$ which contains $p_{\tau}(e)$. Then $\lambda\beta(H)$ must be dense in $DC_e(U_i)$. Thus
$Y_{i+1}$ contains an irreducible component $G$ such that $\pi_i(G)$ is dense in $\mbox{DC}_e(U_i)$, so that $\mbox{DC}_e(U_i)\not\subset E_i$.
\end{proof}

\begin{Lemma}\label{Lemma2} Suppose that $Y$ is a reduced complex analytic space, $e\in \mathcal E_Y$, $\pi_0:Y_0\rightarrow Y\in e$, and  $(U,E,h)$ is a local blow up of $Y_0$. Then
$\pi_0h\in e$ if and only if  $p_{\pi_0}(e)\in U$ and $\mbox{DC}_e(U)\not\subset E$.
\end{Lemma}  

\begin{proof}  The conditions $p_{\pi_0}(e)\in U$ and $\mbox{DC}_e(U)\not\subset E$ are certainly necessary for $\pi_0h$ to be in $e$ (by Lemma \ref{Lemma1}).

Suppose that $p_{\pi_0}(e)\in U$ and $\mbox{DC}_e(U)\not\subset E$. 
We will verify the criterion of Lemma 2.10 on page 431 of \cite{H}.
Suppose that $\pi_{\alpha}:Y_{\alpha}\rightarrow Y\in e$. Let our map $h$ be $h:Y'\rightarrow Y_0$. We must show that there exists $\pi_{\beta}:Y_{\beta}\rightarrow Y\in e$ such that 
$\mbox{Hom}(\pi_{\beta},\pi_{\alpha})\ne \emptyset$, and if $J(\pi_{\beta},\pi_0h):Y_J\rightarrow Y$ is the join, then the natural image of 
$Y_J$ in $Y'$ is relatively compact and non empty.

We have that $U\rightarrow Y$ is in $e$ (by Corollary 2.11.4 \cite{H}), so 
we can replace $Y_0$ with $U$, and assume that $U=Y_0$, and $E$ is closed in $Y_0$ with $\mbox{DC}_e(Y_0)\not\subset E$.
By 2) and 3) of Definition \ref{EtoileDef},
there exists $\pi_{\beta}:Y_{\beta}\rightarrow Y\in e$ and maps $\lambda\in \mbox{Hom}(\pi_{\beta},\pi_0)$,   $\tau\in \mbox{Hom}(\pi_{\beta},\pi_{\alpha})$  such that    $\lambda(Y_{\beta})$ is relatively compact in $Y_{0}$. By the universal property, we have that the joins $J(\pi_{\beta},\pi_0h)$ and $J(\lambda,h)$ are isomorphic, which we will denote by 
$Y_{\beta}'$.  We have a commutative diagram:
$$
\begin{array}{ccccc}
&&Y_{\beta}'&&\\
&\gamma\swarrow&\downarrow&\searrow\delta&\\
Y_{\beta}&\stackrel{\lambda}{\rightarrow}&Y_0&\stackrel{h}{\leftarrow}&Y'\\
\tau\downarrow&\pi_{\beta}\searrow&\pi_0\downarrow&\swarrow h\pi_0&\\
Y_{\alpha}&\stackrel{\pi_{\alpha}}{\rightarrow}&Y&&
\end{array}
$$
Let $K$ be the closure of $\lambda(Y_{\beta})$ in $Y_{0}$, which is compact. $\delta(Y_{\beta}')\subset h^{-1}(\lambda(K))$, which is
compact since $h$ is a global blow up, so it is proper. Thus $\delta(Y_{\beta}')$ is relatively compact.

 It remains to show that
$Y_{\beta}'\ne\emptyset$. We have that $Y_{\beta}\ne\emptyset$ (since $\pi_{\beta}\in e$). The map $\pi_{\beta}$ is strict, by Proposition 1.7 \cite{H}, so it
is an open immersion on an open subset $W$ of $Y_{\beta}$ which intersects $\mbox{DC}_e(Y_{\beta})$ nontrivially. $\lambda$ is thus necessarily also an open immersion on $W$.
Thus $V=\lambda(W)$ is an open subset of $Y_0$ such that $\mbox{DC}_e(Y_0)\cap V\ne\emptyset$. 
By our assumption on $E$, we have that $E\cap\mbox{DC}_e(Y_0)\cap V$ is nowhere dense in $\mbox{DC}_e(Y_0)\cap V$.
Let $F_1,\ldots,F_r$ be the irreducible components of $Y_0$, with $F_1=\mbox{DC}_e(Y_0)$. $h$ is an isomorphism over the non trivial open set
$V\setminus(E\cup F_2\cup \cdots \cup F_r)$. Let $Z= Y_{\beta}|\lambda^{-1}(V\setminus(E\cup F_2\cup \cdots \cup F_r))$.
Let $\epsilon:Z\rightarrow  Y'$ be the morphism induced by $\lambda$ and $i:Z\rightarrow Y_{\beta}$ be the inclusion. Now $Z\ne\emptyset$,
and since $\lambda i=h\epsilon$, we have that $\mbox{Hom}(\epsilon,\delta)\ne \emptyset$ be the universal property of the join.  Thus $Y_{\beta}'\ne \emptyset$.
\end{proof}

Using resolution of singularities, and resolution of indeterminancy (\cite{H1}, \cite{H2}, \cite{AHV},\cite{BM}) we deduce the following Lemma.

\begin{Lemma}\label{Lemma3} Suppose that $Y$ is a reduced complex analytic space and $e\in \mathcal E_Y$. Suppose that $\pi\in e$ factors as a sequence of local blow ups 
$$
Y_n\rightarrow Y_{n-1}\rightarrow \cdots\rightarrow Y_1\rightarrow Y
$$
where each $\pi_i:Y_{i+1}\rightarrow Y_i$ is a local blow up $(U_i,E_i,\pi_i)$. Then there exists $\pi'\in e$ which is a composition of local blow ups 
$$
Y'_n\rightarrow Y'_{n-1}\rightarrow \cdots\rightarrow Y'_1\rightarrow Y
$$
such that  each $Y_i'$ is nonsingular, $\pi_i':Y'_{i+1}\rightarrow Y'_i$ is a local blow up $(U'_i,E'_i,\pi'_i)$ (which is a sequence of blowups with nonsingular centers over $U_i'$), and there exists a commutative diagram of
strict morphisms
$$
\begin{array}{lllllllll}
Y_n'&\rightarrow & Y_{n-1}'&\rightarrow &\cdots &\rightarrow &Y_1'&&\\
\downarrow&&\downarrow&&&&\downarrow&\searrow&\\
Y_n&\rightarrow & Y_{n-1}&\rightarrow &\cdots &\rightarrow &Y_1&\rightarrow &Y.\\
\end{array}
$$
\end{Lemma}

Suppose that $\phi:X\rightarrow Y$ is a morphism of complex analytic spaces, and $\pi:Y'\rightarrow Y\in \mathcal E(Y)$.
$\phi^{-1}[\pi]:\phi^{-1}[Y']\rightarrow X$ will denote the strict transform of $\phi$ by $\pi$ (Section 2 of \cite{HLT}).

In the case of a single local blowup $(U,E,\pi)$ of $Y$, $\phi^{-1}[Y']$ is the blow up $B(\mathcal I_E\mathcal O_{X}|\phi^{-1}(U))$.
In the case when $\pi=\pi_0\pi_{1}\cdots \pi_r$ with $\pi_i:Y_{i+1}\rightarrow Y_{i}$ given by local blow ups $(U_i,E_i,\pi_i)$, we inductively define
$\phi^{-1}[\pi]$. Assume that $\pi^{-1}[\pi_{0}\cdots\pi_{r-1}]$ has been constructed. Let $h=\pi_{0}\cdots\pi_{r-1}$, so that $\pi=h\pi_r$. 
Let $\phi':\phi^{-1}[Y_{r}]\rightarrow Y_{r}$ be the natural morphism. Then define $\phi^{-1}[Y_{r+1}]$ to be the blow up 
$B(\mathcal I_{E_r}\mathcal O_{\phi^{-1}[Y_{r}]}|(\phi')^{-1}(U_r))$.

\begin{Lemma}\label{joinstrict} Suppose that $\pi_1,\pi_2\in\mathcal E(Y)$. Then
$$
J(\phi^{-1}[\pi_1],\phi^{-1}[\pi_2])=\phi^{-1}[J(\pi_1,\pi_2)].
$$
\end{Lemma}

\begin{proof} The fact that these two constructions are canonically isomorphic can be realized by comparing the explicits constructions given above.
The essential case is that of the strict transform of the join of two local blow ups $\pi_1:Y_1\rightarrow Y$ and $\pi_2:Y_2\rightarrow Y$ given by local
 data $(U_1,E_1,\pi_1)$ and $(U_2,E_2,\pi_2)$. The join $J(\pi_1,\pi_2)$ is then the blow up 
 $$
 J(\pi_1,\pi_2):B(\mathcal I_{E_1}\mathcal I_{E_2}|U_1\cap U_2)\rightarrow Y,
 $$
 and $\phi^{-1}[J(\pi_1,\pi_2)]$ is the blow up 
 \begin{equation}\label{eq4}
 \phi^{-1}[J(\pi_1,\pi_2)]B(\mathcal I_{E_1}\mathcal I_{E_2}\mathcal O_X|\phi^{-1}(U_1\cap U_2))\rightarrow X.
 \end{equation} 
 However, $\phi_i^{-1}[\pi_i]$ are the blow ups $\phi_i^{-1}[\pi_i]:B(\mathcal I_{E_i}\mathcal O_X|\phi^{-1}(U_i))\rightarrow X$ Thus the construction of
 $J(\phi^{-1}[\pi_1],\phi^{-1}[\pi_2])$ described at the beginning of this section gives us again the blow up (\ref{eq4}).
 \end{proof}

We now introduce a concept which will play a central role in determining when we can push an \'etoile forward  by a morphism.

\begin{Lemma} Suppose that $\phi:X\rightarrow Y$ is a morphism of  complex analytic spaces and $e\in \mathcal E_X$.
Let 
$$
S(\phi,e)=\{\pi\in \mathcal E(Y)\mid \phi^{-1}[\pi]\in e\}.
$$
Then $S(\phi,e)$ satisfies properties 1), 2) and 3) of Definition \ref{EtoileDef}.
\end{Lemma}

\begin{proof} This follows from Lemma \ref{joinstrict} and 2.9.2 of \cite{H}.
\end{proof}

\begin{Lemma}\label{Lemma6} Suppose that $\phi:X\rightarrow Y$ is a morphism of reduced complex analytic spaces and $e\in \mathcal E_X$. Suppose that $f\in \mathcal E_Y$ contains
$S(\phi,e)$, and $\pi:Y'\rightarrow Y\in S(\phi,e)$. Then
$$
p_{\pi}(f)=\phi'(p_{\phi^{-1}[\pi]}(e))
$$
and
$$
\phi'(\mbox{DC}_e(\phi^{-1}[Y']))\subset \mbox{DC}_f(Y'),
$$
where $\phi':\phi^{-1}[Y']\rightarrow Y'$ is the induced morphism.
\end{Lemma}

\begin{proof} Suppose that  $U$ is any neighborhood of $\phi'(p_{\phi^{-1}[\pi]}(e))$ in $Y'$.
Then $\pi|U\in S(\phi,e)$ (by Lemma \ref{Lemma2}). Thus 
$$
p_{\pi}(f)=\phi'(p_{\phi^{-1}[\pi]}(e)).
$$

Suppose that $\phi'({\rm DC}_e(\phi^{-1}[Y']))\not\subset{\rm DC}_f(Y')$. Then there exists a global blowup $\beta:Y''\rightarrow Y'$ such that ${\rm DC}_f(Y'')$ is a connected component of $Y''$ and the induced morphism $\pi'':Y''\rightarrow Y\in S(\phi,e)$.
 We  have an induced  commutative diagram of analytic morphisms
$$
\begin{array}{ccc}
\phi^{-1}[Y'']&\stackrel{\phi''}{\rightarrow}&Y''\\
\alpha\downarrow&&\downarrow\beta\\
\phi^{-1}[Y']&\stackrel{\phi'}{\rightarrow}&Y'
\end{array}
$$
where the vertical arrow are global blow ups.  By our construction, $\phi''({\rm DC}_e(\phi^{-1}[Y'']))$ is disjoint from ${\rm DC}_f(Y'')$. But
 $p_{\pi''}(f)=\phi''(p_{\phi^{-1}[\pi'']}(e))$ by the first part of this proof, so we have that $\phi''({\rm DC}_e(X''))\subset {\rm DC}_f(Y'')$, a contradiction. Thus
$\phi'({\rm DC}_e(\phi^{-1}[Y'])\subset{\rm DC}_f(Y')$.

\end{proof}

\begin{Proposition}\label{Prop1} Suppose that $\phi:X\rightarrow Y$ is a morphism of reduced complex analytic spaces. Then
$S(\phi,e)\in \mathcal E_Y$ if and only if for all $\pi:Y'\rightarrow Y\in S(\phi,e)$, with associated morphism $\phi':\phi^{-1}[Y']\rightarrow Y'$, $\phi'(\mbox{DC}_e(\phi^{-1}[Y']))$ is not  
contained in a proper analytic subset of an irreducible component of $Y'$.
\end{Proposition}

\begin{proof} Suppose that $f\in \mathcal E_Y$ contains $S(\phi,e)$ and there exists $\pi:Y'\rightarrow Y\in S(\phi,e)$ such that $\phi'(\mbox{DC}_e(\phi^{-1}[Y']))$ is contained
in a proper analytic subset $E$ of an irreducible component of $Y'$. Let $\alpha:Z\rightarrow Y'$ be the blow up of $E$. Then
$\pi\alpha\in f$ by Lemma \ref{Lemma2}. We have a commutative diagram of morphisms
$$
\begin{array}{ccc}
\phi^{-1}[Z]&\stackrel{\phi''}{\rightarrow}&Z\\
\phi^{-1}[\alpha]\downarrow&&\downarrow\alpha\\
\phi^{-1}[Y']&\stackrel{\phi'}{\rightarrow}&Y'.
\end{array}
$$
$\mbox{DC}_e(\phi^{-1}[Y'])$ is a subspace of $(\phi')^{-1}(E)$ and $\phi^{-1}[\alpha]:\phi^{-1}[Z]\rightarrow \phi^{-1}[Y']$ is the blow up of $(\phi')^{-1}(E)$. Thus $\phi^{-1}[\pi\alpha]=\phi^{-1}[\pi]\phi^{-1}[\alpha]\not\in e$ by Lemma \ref{Lemma2}.

Now suppose that for all $\pi:Y'\rightarrow Y\in S(\phi,e)$, with associated morphism $\phi':\phi^{-1}[Y']\rightarrow Y'$, $\phi'(\mbox{DC}_e(\phi^{-1}[Y']))$ is not  
contained in a proper analytic subset of an irreducible component of $Y'$. Suppose that $f\in \mathcal E_Y$ contains $S(\phi,e)$. 
Suppose that $\pi\in f$.
We will show that $\pi\in S(\phi,e)$. 

We prove this by induction on the length $r$ of a factorization $\pi=h_0h_{1}\cdots h_{r-1}h_r$ where $(U_i,E_i,h_i)$ are
local blow ups $h_i:Y_{i+1}\rightarrow Y_i$. By Lemma \ref{Lemma1}, $h_{0}\cdots h_{r-1}\in f$, $p_{h_{0}\cdots h_{r-1}}(f)\in U_r$ and
$\mbox{DC}_f(U_r)\not\subset E_r$.  

We have a commutative diagram of morphisms
$$
\begin{array}{rcl}
\phi^{-1}[Y']&\stackrel{\phi'}{\rightarrow}&Y'\\
\phi_{r-1}^{-1}[h_r]\downarrow&&\downarrow h_r\\
\phi^{-1}[Y_{r}]&\stackrel{\phi_{r-1}}{\rightarrow}&Y_{r}\\
\phi^{-1}[h_{0}\cdots h_{r-1}]\downarrow&&\downarrow h_{0}\cdots h_{r-1}\\
X&\stackrel{\phi}{\rightarrow}&Y.
\end{array}
$$

By our induction assumption, $h_{0}\cdots h_{r-1}\in S(\phi,e)$, so that $\phi^{-1}[h_{0}\cdots h_{r-1}]\in e$.
We have that
$p_{h_0\cdots h_{r-1}}(f)\in U_r$ by Lemma \ref{Lemma1}. Thus $\alpha=h_{0}\cdots h_{r-1}|U_r\in f$ by Lemma \ref{Lemma2}.
$\phi^{-1}[\alpha]:\phi^{-1}[U_r]\rightarrow X$ is in $e$ by Lemma \ref{Lemma6}, since $p_{\phi^{-1}[h_{0}\cdots h_{r-1}]}(e)\in \phi_{r-1}^{-1}(p_{h_0\cdots h_{r-1}}(f))$. Thus $\alpha\in S(\phi,e)$, so that $\phi_{r-1}(\mbox{DC}_e(\phi_{r-1}^{-1}(U_r)))$ is not contained in a
proper analytic subset of an irreducible component of $U_r$, by assumption. Since $\phi_{r-1}(\mbox{DC}_e(\phi_{r-1})^{-1}(U_r))\subset \mbox{DC}_f(U_r)$ by Lemma \ref{Lemma6}, and $\mbox{DC}_f(U_r)\not\subset E_r$ (by Lemma \ref{Lemma1}), we have that
$$
\phi_{r-1}(\mbox{DC}_e(\phi_{r-1}^{-1}(U_r))\not\subset E_r,
$$
 so $\mbox{DC}_e(\phi_{r-1}^{-1}(U_r))\not\subset \phi_{r-1}^{-1}(E_r\cap U_r)$. Thus 
$\phi^{-1}[\pi]=\phi^{-1}[h_{0}\cdots h_{r-1}]\phi_{r-1}^{-1}[h_r]\in e$ by Lemma \ref{Lemma2}, so that $\pi\in S(\phi,e)$.
\end{proof}

\begin{Theorem}\label{Theorempreflat} Suppose that $\phi:X\rightarrow Y$ is a morphism of reduced complex analytic spaces, and $e\in\mathcal E_X$. Then there exists
$\pi:Y'\rightarrow Y\in S(\phi,e)$ (so that $\phi^{-1}[\pi]\in e$) such that either $\phi':\phi^{-1}[Y']\rightarrow Y'$ is flat at $p_{\phi^{-1}[\pi]}(e)$ or $\phi'(\mbox{DC}_e(\phi^{-1}[Y']))$ is contained in a proper analytic subset of an irreducible component of $Y'$. 
\end{Theorem}

\begin{proof} Let $f \in \mathcal E_Y$ be such that $S(\phi,e)\subset f$. By Theorem 3 \cite{HLT} or Theorem 4.4 \cite{H3}, there exists $\pi\in f$ such that
$\phi':\phi^{-1}[Y']\rightarrow Y'$ is flat at points of $(\phi')^{-1}(p_{\pi}(f))\cap (\phi^{-1}[\pi])^{-1}(p_{\mbox{id}}(e))$.
 If $\pi\in S(\phi,e)$, then $\phi^{-1}[\pi]\in e$, so that 
$p_{\phi^{-1}[\pi]}(e)\in(\phi')^{-1}(p_{\pi}(f))\cap (\phi^{-1}[\pi])^{-1}(p_{\mbox{id}}(e))$ by Lemma \ref{Lemma6}, so that $\phi'$ is flat at $p_{\phi^{-1}[\pi]}(e)$.

Now suppose that  $\pi\not\in S(\phi,e)$. We can factor
$\pi=h_0h_{1}\cdots h_r$ where $(U_i,E_i,h_i)$ are
local blow ups $h_i:Y_{i+1}\rightarrow Y_i$. By Lemma \ref{Lemma1}, $h_{0}\cdots h_s\in f$, $p_{h_0\cdots h_s}(f)\in U_r$ and
$\mbox{DC}_f(U_s)\not\subset E_s$ for all $s$. There exists a largest $s$ such that $h_{0}\cdots h_{s-1}\in S(\phi,e)$, but $h_0\cdots h_s\not\in S(\phi,e)$. $U_s\subset Y_{s}$ contains $p_{h_{0}\cdots h_{s-1}}(f)$, so that $U_s\rightarrow Y\in S(\phi,e)$ by Lemma \ref{Lemma2}.

Let $\lambda=(h_{0}\cdots h_{s-1})|U_s$, and
$\phi'':\phi^{-1}[U_s]\rightarrow U_s$ be the induced morphism.
Then $\phi''(\mbox{DC}_e(\phi^{-1}[U_s])\subset \mbox{DC}_f(U_s)$ by Lemma \ref{Lemma6}. Since $\lambda h_s\not\in S(\phi,e)$, we have that 
$$
 \phi''(\mbox{DC}_e(\phi^{-1}[U_s])\subset E_s\cap \mbox{DC}_f(U_s),
 $$
  which is a proper analytic subset of the irreducible component $\mbox{DC}_f(U_s)$ of $U_s$. Now replacing $\pi$ with $\lambda$, we have obtained the conclusions of the theorem.
\end{proof}

\begin{Corollary}\label{flat} Suppose that $\phi:X\rightarrow Y$ is a morphism of reduced  complex analytic spaces, and $e\in\mathcal E_X$. Then there exists a commutative diagram of morphisms
$$
\begin{array}{lcr}
\tilde X&\stackrel{\tilde \phi}{\rightarrow}&\tilde Y\\
\gamma\downarrow&&\downarrow\delta\\
X&\stackrel{\phi}{\rightarrow}&Y
\end{array}
$$
such that 
$\gamma\in e$, $\delta$ is sequence of morphisms consisting of local blow ups and inclusions of proper analytic subsets,
$\tilde X$ is reduced, $\tilde Y$ is reduced, and
$\tilde \phi$ is flat at $p_{\gamma}(e)$.
 \end{Corollary}
 
 \begin{proof} The proof is by induction on the dimension of $Y$. If $\dim Y=0$, then $Y$ is a finite union of points, so $\phi$ is necessarily flat, since $\mathcal O_{Y,q}$
 is a field for all $q\in Y$. Suppose that the Corollary is true for all reduced complex analytic spaces of dimension less than
 $\dim Y$.

 By Theorem \ref{Theorempreflat}, there exists $\pi:Y'\rightarrow Y\in S(\alpha,e)$ such that
 either 
 \begin{equation}\label{eqflat1}
 \mbox{the induced morphism $\phi':\phi^{-1}[Y']\rightarrow Y'$ is flat at $p_{\phi^{-1}[\pi]}(e)$,}
 \end{equation}
  or
  \begin{equation}\label{eqflat2}
  \mbox{$\phi'(\mbox{DC}_e(\phi^{-1}[Y']))$
 is contained in a proper analytic subset of an irreducible component of $Y'$.}
 \end{equation}

 If (\ref{eqflat1}) holds then we have achieved the conclusions of the Corollary. Suppose that (\ref{eqflat2}) holds. 
 There exists an irreducible analytic subset $F$ of $Y'$ such that $\phi'(\mbox{DC}_e(\pi^{-1}[Y']))\subset F$
 and $F$ is not an irreducible component of $Y'$ (so that $\dim F<\dim Y$).
 
 Let $\tau:X''\rightarrow \phi^{-1}[Y']$ be a resolution of singularities, obtained by blowing up a nowhere dense closed  analytic subspace of $\phi^{-1}[Y']$.
 Then $\phi^{-1}[\pi]\tau\in e$. Then $X^*=\mbox{DC}_e(X'')$ is a connected component of $X''$, so the composition of inclusion of $X^*$ into $X''$
 and the morphism $\phi^{-1}[\pi]\tau$ is in $e$. We have an induced morphism of $X^*$ to $F$. By induction on the dimension of $Y$,
 the conclusions of the Corollary hold.
 \end{proof}
 
\begin{Proposition}\label{reg3} Suppose that $\phi:X\rightarrow Y$ is a morphism of reduced, irreducible, locally irreducible complex analytic spaces and  $\phi$ is regular. Further suppose that  $\alpha:X'\rightarrow X$, $\beta:Y'\rightarrow Y$ are sequences of local blow ups 
such that $X'$ and $Y'$ are reduced, irreducible, locally irreducible, and there is a commutative diagram of morphisms
$$
\begin{array}{lcr}
X'&\stackrel{\phi'}{\rightarrow}&Y'\\
\alpha\downarrow&&\downarrow\beta\\
X&\stackrel{\phi}{\rightarrow}&Y.
\end{array}
$$
 Then $\phi'$ is regular.
\end{Proposition}

\begin{proof} There exists an analytic subset $F$ of $Y'$ such that $\dim F<\dim Y'=\dim Y$, $\dim \beta(F)<\dim Y$, $Y'\setminus F=Y'\setminus \beta^{-1}(\beta(F))$,
and  $\beta|(Y'\setminus F):Y'\setminus F\rightarrow Y$ is an isomorphism onto an open subset of $Y$.

There exists an analytic subset $H$ of $X'$ such that $\dim H<\dim X'=\dim X$, $\dim \alpha(H)<\dim X$, $V=X'\setminus H=X'\setminus \alpha^{-1}(\alpha(H))$
is an open subset of $X'$
and  $\alpha|V:V\rightarrow X$ is an isomorphism onto an open subset of $X$.

Since $\phi$ is regular, by Lemma \ref{Lemma5}, there exists a nowhere dense closed analytic subset $G$ of $X$ such that $\phi(X\setminus G)$ is an open subset of $Y$,  $\dim \phi(G)<\dim Y$, and $\phi|(X\setminus G)$ is an open mapping.

$\dim G\le \dim X-1$ implies $W:=(X\setminus G)\cap\alpha(V)$ is a nonempty open subset of $X$. $\phi(W)$ is an open subset of $Y$.
$\phi(W)\subset \beta(Y')=\beta(Y'\setminus F)\cup \beta(F)$. Since $\dim \beta(F)<\dim Y$, we have that $\phi(W)\cap \beta(Y'\setminus F)$ is a nonempty open subset of $Y$. Since $\alpha$ is an isomorphism over $W$ and $\beta$ is an isomorphism over $\beta(Y'\setminus F)$, we have that  $\phi'(V)$ contains the nonempty open set $\beta^{-1}(\phi(W)\cap \beta(Y'\setminus F))$. Thus $\phi'$ is regular.
\end{proof}

 \begin{Proposition}\label{Prop2} Suppose that $\phi:X\rightarrow Y$ is a morphism of reduced, irreducible, locally irreducible complex analytic spaces, and $e\in \mathcal E_X$.
 Then $S(\phi,e)\in \mathcal E_Y$ if and only if $\phi$ is regular.
 \end{Proposition}
 
 \begin{proof} 
 In Corollary \ref{flat}, $\tilde\phi$ is an open morphism to $\tilde Y$, since $\tilde\phi$ is flat (\cite{D} or Theorem V.2.12 \cite{BS}) so $\phi$ is regular if and only if $\delta$ is a sequence of local blowups. Thus the Proposition follows from Proposition \ref{Prop1}, and since a local blow up is strict. 
 \end{proof}
 
 \begin{Proposition}\label{Etoileext} Suppose that $\phi:X\rightarrow Y$ is a morphism of reduced, irreducible, locally irreducible complex analytic spaces and $f\in \mathcal E_Y$. Then there exists $e\in \mathcal E_X$ such that $S(\phi,e)=f$ if and only if $\phi$ is regular.
 \end{Proposition}
 
 \begin{proof} By Proposition \ref{Prop2}, if such an $e$ exists then $\phi$ must be regular, so suppose that $\phi$ is regular. We may restrict $\phi$ to a relatively compact open subset of $X$. Let $e_0$ be the subcategory of $\mathcal E(X)$ of morphisms of analytic spaces determined by the associated morphisms  $\phi^{-1}[Y_1]\rightarrow X$ of $\pi:Y_1\rightarrow Y\in f$. Since $\phi$ is regular and $\pi$ is strict, $\phi^{-1}[Y_1]\ne \emptyset$ for all $\pi\in f$ so $e_0$ satisfies 1) of Definition \ref{EtoileDef} of an \'etoile. Since $f$ is an \'etoile, $e_0$ satisfies 2) and 3) of the definition of an \'etoile. By Zorn's lemma, there exists a \'etoile $e\in \mathcal E_X$ containing $e_0$ (Lemma 2.2 \cite{H}).
 
 Now $f\subset S(\phi,e)$ and $S(\phi,e)$ is an \'etoile on $Y$ (by Proposition \ref{Prop2}) so $f=S(\phi,e)$ since $f$ satisfies the maximality condition 4) of Definition \ref{EtoileDef}.
 \end{proof}

\section{Regularization of analytic maps}\label{Secreg}

\begin{Theorem}\label{regmap} Suppose that $\phi:X\rightarrow Y$ is a morphism of reduced  complex analytic spaces, and $e\in\mathcal E_X$. Then there exists a commutative diagram of morphisms
$$
\begin{array}{lcr}
\tilde X&\stackrel{\tilde \phi}{\rightarrow}&\tilde Y\\
\gamma\downarrow&&\downarrow\delta\\
X&\stackrel{\phi}{\rightarrow}&Y
\end{array}
$$
such that 
$\gamma\in e$, $\delta$ is sequence of morphisms consisting of local blow ups  and inclusions of proper analytic subsets,
$\tilde X$ is nonsingular and irreducible, $\tilde Y$ is nonsingular and irreducible and
$\tilde \phi$ is regular.
 \end{Theorem}
 
 \begin{proof} By Corollary \ref{flat}, we may assume that $\phi$ is flat.
  Let $p=p_{\mbox{id}}(e)$.

 There exists an open subset $V$ of $Y$ which contains $\phi(p)$, such that all irreducible components of $V$ are locally irreducible.

 There exists an open subset $U$ of $\phi^{-1}(V)$ containing $p$ such that $\mbox{DC}_e(U)$ is locally irreducible.
 Let $G$ be the union of the irreducible components of $U$ other than ${\rm DC}_e(U)$.
 Let
 $W=\mbox{DC}_e(U)\setminus G$. $W$ is a nonempty open subset of $X$, so $\phi(W)$ is an open subset of $V$, since $\phi$ is flat, \cite{D} or Theorem V.2.12 \cite{BS}.
 Let $V^*$ be the irreducible component of $V$ containing $\phi(\mbox{DC}_e(U))$.
 By definition, the induced map $\phi:\mbox{DC}_e(U)\rightarrow V^*$ is regular at $p$.

 Let $\tau:V'\rightarrow V$ be a resolution of singularities. 
  $\tau$ is the blow up of a nowhere dense closed analytic set $E$, and $H=\phi^{-1}(E)$ is nowhere dense in $X$ since 
 $\phi$ is flat. Let $\mathcal I$ be the ideal sheaf of $H$ in $X$.

 Let $\pi:X'\rightarrow U$ be a resolution of singularities, obtained by a sequence of global blow ups of nowhere dense closed analytic sets, so that
 $\mathcal I\mathcal O_{X'}$ is invertible. The composition of $\pi$ with the inclusion of $U$ into $X$ is in $e$. Since $\mbox{DC}_e(X')$ is a connected component of $X'$, the induced morphism $\mbox{DC}_e(X')\rightarrow X$ is in $e$.
 
 $\mbox{DC}_e(X')$ is necessarily the strict transform of $\mbox{DC}_e(U)$ in $X'$. Thus $\mbox{DC}_e(X')\rightarrow \mbox{DC}_e(U)$
 is a product of blow ups. Thus the induced morphism $\tilde X=\mbox{DC}_e(X')\rightarrow \tilde Y=(V')^*$
 is regular 
 by Proposition \ref{reg3}, where $(V')^*$ is the connected component of $V'$ which contains the image of  $\mbox{DC}_e(X')$.
 \end{proof}

\begin{Theorem}\label{TheoremA} Suppose that $\phi:Y\rightarrow X$ is a morphism of reduced  complex analytic spaces and $e\in \mathcal E_Y$
is an \'etoile over $Y$. Then there exists a commutative diagram of morphisms
\begin{equation}\label{eq1}
\begin{array}{ccc}
\tilde Y&\stackrel{\tilde \phi}{\rightarrow}&\tilde X\\
\delta\downarrow &&\downarrow \gamma\\
Y&\stackrel{\phi}{\rightarrow}& X
\end{array}
\end{equation}
such that $\delta\in e$ is a finite product of local blow ups of  nonsingular analytic sub varieties, $\gamma$ is a finite product of local blow ups of  nonsingular analytic sub varieties    and inclusions of analytic sub varieties, $\tilde Y$ and $\tilde X$ are smooth analytic spaces and $\tilde \phi$ is a regular analytic morphism.
\end{Theorem}
 
 \begin{proof} The proof is by induction on the dimension of $Y$. When $Y$ has dimension zero then 
 letting $\tilde X=\phi(p_{\rm id}(e))$ and $\tilde Y$ be the connected component of a resolution of singularities of $Y$ (which is a product of blowups of nonsingular nowhere dense sub varieties)
 we have that $\tilde \phi$ is regular. 
 
 Suppose that $\dim Y=r$ and the theorem is true when $Y$ has dimension less than $r$. Let
 $$
 \begin{array}{rcl}
 \tilde X&\stackrel{\tilde \phi}{\rightarrow}&\tilde Y\\
 \gamma\downarrow&&\downarrow\delta\\
 X&\stackrel{\phi}{\rightarrow}&Y
 \end{array}
 $$
 be the diagram constructed in Theorem  \ref{regmap}.
 
 We can  factor the diagram as
 $$
 \begin{array}{rcl}
 \tilde X&\stackrel{\tilde\phi}{\rightarrow}&\tilde Y\\
 \gamma_2\downarrow&&\downarrow \delta_2\\
 X_1&\stackrel{\phi_1}{\rightarrow}&Y_1\\
 \gamma_1\downarrow&&\downarrow\delta_1\\
 X&\stackrel{\phi}{\rightarrow}&Y
 \end{array}
 $$
 where $\delta_1\in S(\phi,e)$,  $\gamma_1\in e$,   
 and either $X_1=\tilde X$ and  $Y_1=\tilde Y$ or 
 $$
 \phi_1({\rm DC}_e(X_1))\mbox{ is contained in a proper analytic subset of $Y_1$.}
 $$
   Factor $\delta_1:Y_1\rightarrow Y$ as 
 \begin{equation}\label{eqF2}
Y_1= Z_c\stackrel{\alpha_c}{\rightarrow} Z_{c-1}\stackrel{\alpha_{c-1}}{\rightarrow}\cdots\stackrel{\alpha_1}{\rightarrow} Z_0=Y
 \end{equation}
 where each $\alpha_i$ is a local blow up. 
 
 The morphism $\alpha_1$ is the blow up of an analytic subspace $E_0$ of an open  subset $U_0$ of $Y$. 
 By principalization of ideals, there exists a sequence of blow ups of nonsingular analytic subspaces $W_1\rightarrow U_0$ such that the ideal sheaf $\mathcal I_{E_0}\mathcal O_{W_1}$ is locally principal. Then by the universal property of blow ups of ideals, there is a factorization $W_1\stackrel{\sigma_1}{\rightarrow}Z_1\rightarrow Y$.
 Factor the proper map $W_1\rightarrow U_0$ as a sequence of blow ups of nonsingular analytic sub varieties
 \begin{equation}\label{eqF1}
 W_1=V_s\stackrel{\tau_s}{\rightarrow}V_{s-1}\rightarrow \cdots\rightarrow V_1\stackrel{\tau_1}{\rightarrow} U_0.
 \end{equation}
 
 Suppose there exists an index $t$ in (\ref{eqF1}) such that $V_t\rightarrow Y\in S(\phi,e)$ but $V_{t+1}\rightarrow Y\not\in S(\phi,e)$. This can only happen if the image of ${\rm DC}_e(\phi^{-1}[V_t])$ in $V_t$ is contained in the analytic subspace $F_t$ of $V_t$ blown up in $V_{t+1}\rightarrow V_t$. We  have a commutative diagram
 $$
 \begin{array}{rcl}
 \phi^{-1}[V_t]&\stackrel{\phi'}{\rightarrow}&V_t\\
 \alpha\downarrow&&\downarrow\beta\\
 X&\stackrel{\phi}{\rightarrow}&Y
 \end{array}
 $$
 such that $\beta$ is a sequence of local blow ups of nonsingular analytic sub varieties and $\phi^{-1}[V_t]\rightarrow X\in e$.
 
  Let $\lambda:X''\rightarrow X\in e$ be a sequence of local blow ups of nonsingular analytic sub varieties such that $X''$ is nonsingular and there is a factorization $X''\rightarrow \phi^{-1}[V_t]\rightarrow X$.
  Let $X^*={\rm DC}_e(X'')$. The sub variety $X^*$ is a connected component of $X''$ since $X''$ is nonsingular, so that composition of the inclusion of $X^*$ into $X''$ and the morphism $\lambda$ is in $e$. We have an induced morphism $X^*\rightarrow F_t$.  The theorem now follows by induction on $\dim Y$, as $\dim F_t<r=\dim Y$.

  Now suppose that $W_1\rightarrow Y\in S(\phi,e)$.  Recall that $\sigma_1:W_1\rightarrow Z_1$ is the induced morphism. The local blow up $\alpha_2:Z_2\rightarrow Z_1$ in (\ref{eqF2}) is the blow up of an analytic subspace $E_1$ of an open subset $U_1$ of $Z_1$.  We now construct a sequence of blow ups of nonsingular analytic sub varieties $W_2\rightarrow \gamma_1^{-1}(U_1)$ such that $\mathcal I_{E_1}\mathcal O_{W_2}$ is a locally principal ideal sheaf. We either have that the composition $W_2\rightarrow W_1\rightarrow Y\not\in S(\phi,e)$, in which case we obtain, as explained above,  a reduction in the dimension of $Y$ from which the theorem follows, or we obtain $W_2\rightarrow W_1\rightarrow Y\in S(\phi,e)$ which is a composition of local blow ups of nonsingular analytic  subspaces.

 Continuing in this way, we either obtain a reduction to $\dim Y<r$, from which the theorem follows, or we construct a morphism $\epsilon:W\rightarrow Y_1$ 
  such that $\lambda=\delta_1\epsilon\in S(\phi,e)$ is a composition of local blow ups of nonsingular sub varieties. We have the induced diagram
 $$
 \begin{array}{rcl}
 &&W\\
 &&\downarrow\epsilon\\
 X_1&\stackrel{\phi_1}{\rightarrow}&Y_1\\
 \gamma_1\downarrow&&\downarrow\delta_1\\
 X&\stackrel{\phi}{\rightarrow}&Y.
 \end{array}
 $$
  By resolution of singularities and principalization of ideals, there exists a sequence of local blow ups of nonsingular analytic sub varieties $\psi:X^*\rightarrow X\in e$ such that $X^*$ is nonsingular and connected and there is a commutative diagram of morphisms 
 $$
 \begin{array}{rcl}
 &X^*&\\
 \swarrow&&\searrow\\
 X_1&\downarrow \psi&\phi^{-1}[W]\\
 \searrow&&\swarrow\\
 &X.&
 \end{array}
 $$
 
 First suppose that $\phi_1:X_1\rightarrow Y_1$ is regular.  Let $\xi:W^*\rightarrow W$ be a sequence of blow ups of nonsingular analytic sub varieties which are nowhere dense such that $W^*$ is nonsingular. Then $W^*\rightarrow W\rightarrow Y\in S(\phi,e)$ by Proposition \ref{Prop2}. Let $X^{**}\rightarrow X\in e$ be a sequence of blow ups of nonsingular analytic sub varieties such that there is a commutative diagram
 $$
 \begin{array}{rcl}
 X^{**}&\rightarrow &W^*\\
 \downarrow&&\downarrow\xi\\
 X^*&\rightarrow& W.
 \end{array}
 $$
 The morphism $X^{**}\rightarrow W^*$ is regular by Proposition \ref{reg3}, and we have obtained the conclusions of Theorem \ref{TheoremA}.
 
 Now suppose that $\phi_1({\rm DC}_e(X_1))$ is contained in a nowhere dense analytic subspace $G$ of $Y_1$. Then the image of $X^*$ in $W$ is contained in the preimage
 $Y^*$ of $G$ in $W$, which is nowhere dense in $W$. The theorem now follows from induction on $r$ with the morphism $X^*\rightarrow Y^*$ since $\dim Y^*<r$.

 \end{proof}

\begin{Theorem}\label{TheoremRegFin} Suppose that $\phi:Y\rightarrow X$ is a morphism of reduced  complex analytic spaces and $e\in \mathcal E_Y$
is an \'etoile over $Y$. Then there exists a commutative diagram of morphisms
$$
\begin{array}{ccc}
Y_e&\stackrel{\phi_e}{\rightarrow}&X_e\\
\delta\downarrow &&\downarrow \gamma\\
Y&\stackrel{\phi}{\rightarrow}& X
\end{array}
$$
such that $\delta\in e$ is a finite product of local blow ups of nonsingular analytic sub varieties, $\gamma$ is a finite product of local blow ups of nonsingular analytic sub varieties, $Y_e$ and $X_e$ are smooth analytic spaces and $ \phi_e$ is a regular analytic morphism to a nonsingular analytic sub variety of $X_e$.
\end{Theorem}

\begin{proof} The Theorem follows from Theorem \ref{TheoremA} and the observation that if $W$ is an analytic space, $Z\subset W$ is a closed analytic subspace and $V\subset Z$ is a closed analytic subspace, then the blow up of $V$ in $Z$ is the strict transform of $Z$ in the blow up of $V$ in $W$.
\end{proof}

 \section{The valuation associated to an \'etoile}\label{Secval}
 
 Suppose that $Y$ is a reduced complex analytic space, $e\in\mathcal E_Y$ and $\pi\in e$. We will call $\pi$ nonsingular if $\pi$ is a composition of
 local blow ups
 $$
 Y_{n}\rightarrow Y_{n-1}\rightarrow \cdots\rightarrow Y_1\rightarrow Y
 $$
 such that each $Y_i$ is nonsingular.
 
 We associate to a nonsingular $\pi\in e$ the local ring $A_{\pi}=\mathcal O_{X,p_{\pi}(e)}$. The set 
 $$
 \{A_{\pi}\mid \pi\in e\mbox{ is nonsingular}\}
 $$
 is then a directed set, by Lemma \ref{Lemma3} and Definition \ref{EtoileDef}.  
 The set of quotient fields $K_{\pi}$ of the $A_{\pi}$ also form a
 directed set.  Let
 $$
 \Omega_e=\lim_{\rightarrow}K_{\pi}\mbox{ and }V_e=\lim_{\rightarrow}A_{\pi}.
 $$
  $\Omega_e$ is a field, and $V_e$ is a local ring with quotient field $\Omega_e$.
 
 \begin{Lemma}\label{Lemma4} $V_e$ is a valuation ring.
 \end{Lemma}
 \begin{proof} Suppose that $f\in K_e$. Then there exists $\pi\in e$ such that $f\in K_{\pi}$. $f=\frac{g}{h}$ with $g,h\in A_{\pi}$, where 
 $A_{\pi}$ is the local ring associated to $\pi:X_{\pi}\rightarrow Y$. Let $U\subset X_{\pi}$ be an open neighborhood of $p_{\pi}(e)$ on which 
 $g$ and $h$ are  holomorphic. 
 There exists an ideal sheaf $\mathcal I\subset \mathcal O_{U}$ such that the blow up $X'=B(\mathcal I)$ of $\mathcal I$ is nonsingular, and $(g,h)\mathcal O_{X'}$ is locally principal. Let $\lambda:X'\rightarrow X_{\pi}$ be the induced local blow up. $\pi\lambda\in e$ by Lemma \ref{Lemma2}. We have that either $g\mid h$ or $h\mid g$ in $A_{\pi\lambda}$. Thus $f$ or $\frac{1}{f}\in A_{\pi\lambda}\subset V_e$.
 \end{proof}
 
 \begin{Proposition}\label{Prop3} Suppose that $X$ and $Y$ are reduced, irreducible and locally irreducible complex analytic spaces,  $\phi:X\rightarrow Y$ is a regular morphism,
 and $e\in \mathcal E_X$. Then $f=S(\phi,e)\in \mathcal E_Y$, $\Omega_f\subset\Omega_e$ and $V_f=V_e\cap \Omega_f$.
 \end{Proposition}
 
 If $\phi:X\rightarrow Y$ is not regular, then by Proposition \ref{Prop2}, the valuation ring $V_e$ associated to an \'etoile $e$ on $Y$ does not induce an \'etoile on $Y$ and does not induce an associated valuation ring. 
 
 \begin{proof} $f\in\mathcal E_Y$ by Proposition \ref{Prop2}.
 
 Suppose that $\pi:Y'\rightarrow Y\in f$ is nonsingular. By Lemma \ref{Lemma3}, there exists a nonsingular $\alpha:Z\rightarrow X\in e$ such that
 $\mbox{Hom}(Z,\phi^{-1}[Y'])\ne \emptyset$. We have associated local homomorphisms
 \begin{equation}\label{eq6}
 \mathcal O_{Y',p_{\pi}(f)}\stackrel{(\phi')^*}{\rightarrow} \mathcal O_{\phi^{-1}[Y'],p_{\phi^{-1}[\pi]}(e)}
 \rightarrow \mathcal O_{Z,p_{\alpha}(e)}
 \end{equation}
 where $\phi':\phi^{-1}[Y']\rightarrow Y'$ is the natural  morphism. By Proposition \ref{reg3}, the homomorphism of the sequence of complete local rings
 \begin{equation}\label{eq7}
 \hat{\mathcal O}_{Y',p_{\pi}(f)}\stackrel{\widehat{(\phi')^*}}{\rightarrow} \hat{\mathcal O}_{\phi^{-1}[Y'],p_{\phi^{-1}[\pi]}(e)}
 \rightarrow \hat{\mathcal O}_{Z,p_{\alpha}(e)}
 \end{equation}
 is 1-1. Thus the homomorphism in (\ref{eq6}) is 1-1. We have an associated inclusion of rings $A_{\pi}\rightarrow A_{\alpha}$ with 
 induced inclusion of quotient fields $K_{\pi}\rightarrow K_{\alpha}$. This gives us 1-1 homomorphisms $A_{\pi}\rightarrow V_e$ and $K_{\pi}\rightarrow \Omega_e$.

 Taking the limit over the nonsingular elements of $f$,  we have natural 1-1 homomorphisms $V_f\rightarrow V_e$ and $\Omega_f\rightarrow \Omega_e$.

Suppose that $h\in \Omega_f\cap V_e$. Then there exist nonsingular $\alpha:Y_1\rightarrow Y\in f$ and  $\beta:X_1\rightarrow X\in e$ such that $h\in A_{\beta}\cap K_{\alpha}$.
$h$ has an expression $h=\frac{a}{b}$ with $a,b\in A_{\alpha}$.  Let $U$ be a neighborhood of $p_{\alpha}(f)$ on which $a$ and $b$ are analytic.
There exists $\gamma:Y_2\rightarrow Y$ such that $Y_2$ is the blow up of an ideal sheaf of $\mathcal O_U$, $Y_2$ is nonsingular and $(f,g)\mathcal O_{Y_2}$
is locally principal. Thus $\gamma\in f$ is nonsingular, and either $h$ or $\frac{1}{h}\in A_{\gamma}$. 
There exists $\delta:X_2\rightarrow X$
in $e$ which is nonsingular, such that $\mbox{Hom}(X_2,\phi^{-1}[Y_2])\ne\emptyset$ and $\mbox{Hom}(X_2,X_1)\ne \emptyset$. We have constructed a commutative diagram:
$$
\begin{array}{ccccccccc}
&&&&&&X_2&&\\
&&&&&\swarrow&\delta\downarrow&\searrow&\\
Y_1&\leftarrow&Y_2&\leftarrow&\phi^{-1}[Y_2]&&&&X_1\\
\alpha\searrow&&\swarrow\gamma&&&\searrow&&\swarrow\beta&\\
&Y&&&&&X&&\end{array}
$$

If $h\in A_{\gamma}$ then $h\in V_f$. Suppose that $h\not\in A_{\gamma}$. Then $\frac{1}{h}\in \mathfrak m$ where $\mathfrak m$ is the maximal ideal of
$A_{\gamma}$. Now $A_{\gamma}\rightarrow A_{\delta}$ is a local homomorphism, so $\frac{1}{h}$ is in the maximal ideal $\mathfrak n$ of $A_{\delta}$.
But this is impossible since $h\in A_{\delta}$. Thus we must have $h\in A_{\gamma}\subset V_e$.
 \end{proof}

 Suppose that $Y$ is a reduced  complex space and $e\in\mathcal E_Y$. Let $V_e$ be the valuation ring associated to $e$.
 We have a directed system $\{A_{\pi}\}$ for $\pi:Y'\rightarrow Y\in e$, where we define $A_{\pi}=\mathcal O_{\mbox{DC}_e(Y'),p_{\pi}(e)}$.
 In the case when $\pi$ is nonsingular, $Y'=\mbox{DC}_e(Y')$, so this agrees with our earlier definition.  Taking the limit over this larger directed system again gives us the same limit $V_e$, by Lemma \ref{Lemma3}.

 For  $\pi:Y_0\rightarrow Y\in e$, let 
  $V_{\pi}=V_e\cap K_{\pi}$, which is a valuation ring of $K_{\pi}$, which dominates $A_{\pi}$. Let $\mathfrak m_{\pi}$ be the maximal ideal of $A_{\pi}$.
  Let $\nu_e$ be a valuation of $\Omega_e$ whose valuation ring is $V_e$, and let $\nu_{\pi}$ be the restriction of $\nu$ to $K_{\pi}$, so that $V_{\pi}$ is the valuation ring of $\nu_{\pi}$.

 \begin{Lemma}\label{ratrank} $V_e$ has finite rational rank, which is less than or equal to  $\dim Y$.
 \end{Lemma}
 
 \begin{proof} Suppose that $V_e$ has rational rank larger than $n=\dim Y$. Choose $t_1,\ldots, t_{n+1}\in V_e$ such that their values are rationally independent.
 There exists a nonsingular $\pi:Y'\rightarrow Y\in e$ such that $t_1,\ldots, t_{n+1}\in V_{\pi}$. Thus $\nu_{\pi}$ has rational rank $>\dim Y$. But
 $\nu_{\pi}$ dominates the  noetherian local domain $\mathcal O_{Y',p_{\pi}(e)}$, which has dimension $\le \dim Y$. This is a contradiction to Abhyankar's inequality \cite{Ab}, Appendix 2 of \cite{ZS}.
 \end{proof}

  Thus the rank $r$ of $V_e$ is finite (by Lemma \ref{ratrank}), with 
  $$
  r=\mbox{rank}(V_e)\le \mbox{ratrank}(V_e)\le \dim Y.
  $$

  \begin{Theorem}\label{TheoremAbh} Suppose that $Y$ is a reduced complex analytic space and $e\in \mathcal E_Y$.  Let $V_{e}$ be the valuation ring associated to $e$ and suppose that $V_e$ has maximal rational rank equal to the dimension $n$ of $Y$. Then  the value group of $V_{e}$  is isomorphic to $\ZZ^n$ (as an unordered group).
  \end{Theorem}
  
  To prove Theorem \ref{TheoremAbh}, we require the following two Lemmas, which follow from the very nice properties of Abhyankar valuations.
  
  \begin{Lemma}\label{LemmaAbh1}
  Suppose that $R$ is  an equidimensional regular local ring of dimension $n$ and $\omega$ is a valuation of the quotient field of $R$ which dominates $R$ and has rational rank $n$. Suppose that $R$ has a regular system of parameters $x_1,\ldots,x_n$ such that $\omega(x_1),\ldots,\omega(x_n)$ generate the value group $\Gamma_{\omega}$ of $\omega$. 
  
  Then there exists a unique extension $\hat\omega$ of $\omega$ to a valuation of the quotient field of the $m_R$-adic completion $\hat R$ of $R$  which dominates $\hat R$, and its value group $\Gamma_{\hat\omega}=\Gamma_{\omega}$.
  \end{Lemma}
  
  \begin{proof} Let $k$ be a coefficient field of $\hat R$, so that $\hat R=k[[x_1,\ldots,x_n]]$. The unique extension $\hat\omega$ of $\omega$ is then defined by
  $$
  \hat\omega(f)=\min\{i_1\omega(x_1)+\cdots+i_n\omega(x_n)\mid a_{i_1,\ldots,i_n}\ne 0\}
  $$
  for $f=\sum a_{i_1,\ldots,i_n}x_1^{i_1}\cdots x_n^{i_n}\in \hat R$
  with $a_{i_1,\ldots,i_n}\in k$.
  \end{proof}
  
  \begin{Lemma}\label{LemmaAbh2} Suppose that $R$ is  an equidimensional regular local ring of dimension $n$ and $\omega$ is a valuation of the quotient field of $R$ which dominates $R$ and has rational rank $n$. Suppose that $R$ has a regular system of parameters $x_1,\ldots,x_n$ such that $\omega(x_1),\ldots,\omega(x_n)$ generates the value group $\Gamma_{\omega}$ of $\omega$.  Suppose that $I$ is an ideal in $R$. Then there exist monomials $M_1,\ldots, M_r$ in $x_1,\ldots,x_n$ such that 
  $$
  S=R\left[\frac{M_2}{M_1},\ldots,\frac{M_r}{M_1}\right]_{m_\omega\cap R\left[\frac{M_2}{M_1},\ldots,\frac{M_r}{M_1}\right]}
  $$
  is a regular local ring with regular parameters $y_1,\ldots,y_n$ such that there exist a matrix $A=(a_{ij})$ of natural numbers with determinant $\pm1$ such that
  $$
  x_i=\prod_{j=1}^ny_j^{a_{ij}}\mbox{ for }1\le i\le n
  $$
  and $R\rightarrow S$ factors as a finite sequence of local blow ups of nonsingular sub varieties, whose ideals are generated by Laurent  monomials in $x_1,\ldots,x_n$,
  such that $IS$ is a principal ideal.
  
  \end{Lemma}  
  
  \begin{proof} This follows from the maximal rank case of Zariski's proof of embedded local uniformization \cite{LU}.
  \end{proof}
  \vskip .2truein
  We now prove Theorem \ref{TheoremAbh}. 
  
  Let $f_1,\ldots,f_n\in V_e$ be such that $\nu_e(f_1),\ldots,\nu_e(f_n)$ are rationally independent. There exists a nonsingular $\sigma\in e$ such that $f_1,\ldots,f_n\in A_{\sigma}$.
  Let $\nu_{\sigma}$ be the restriction of $\nu_e$ to the field $K_{\sigma}$. Then $\nu_{\sigma}$ has rational rank $n=\dim A_{\sigma}$, so by Abhyankar's Theorem \cite{Ab}, the value group $\Gamma=\Gamma_{\sigma}$
  of $\nu_{\sigma}$ is isomorphic to $\ZZ^n$ as an unordered group. Let $g_1,\ldots,g_n\in V_{\sigma}$ be such that $\nu_{\sigma}(g_1),\ldots,\nu_{\sigma}(g_n)$ generate $\Gamma$, and let 
  $A_{\sigma}\rightarrow B$ be a sequence of algebraic blow ups along $\nu_{\sigma}$ of  regular prime ideals  such that the regular local ring $B$ has regular parameters $x_1,\ldots,x_n$ such that 
  each $g_i$ is a monomial in $x_1,\ldots,x_n$ (this is possible for instance by \cite{LU}). Then $\nu_{\sigma}(x_1),\ldots,\nu_{\sigma}(x_n)$ are a free basis of the unordered group 
  $\Gamma_{\sigma}$.
  
  By our construction, there exists a nonsingular $\lambda:Z\rightarrow Y$ in $e$ such that $A_{\lambda}=B^{\rm an}$.  By Lemma \ref{LemmaAbh1}, we have that the value group
  $\Gamma_{\lambda}$ of $\nu_{\lambda}=\nu_e|K_{\lambda}$ is $\Gamma$, and $A_{\lambda}$ has a regular system of parameters $x_1,\ldots,x_n$ such that $\nu(x_1),\ldots,\nu(x_n)$ is a free basis of the group $\Gamma$.

  Suppose that $\pi\in e$ is nonsingular. Then $\pi$ has a factorization by local blow ups  
  $$
  Y'=Y_r\stackrel{\pi_{r-1}}{\rightarrow} \cdots \stackrel{\pi_1}{\rightarrow} Y_1\stackrel{\pi_0}{\rightarrow} Y_0=Y
  $$
  where each $\pi_i$ is the blow up of a closed analytic subspace $E_i$ of an open neighborhood $U_i$ of $e_{Y_{i-1}}$ and each $Y_i$ is nonsingular (for $i>0$).
  
  We will show that there exists a nonsingular $Z'\rightarrow Z\rightarrow Y\in e$ which has a factorization
  $$
  Z'=Z_r\stackrel{\tau_{r-1}}{\rightarrow} \cdots \stackrel{\tau_1}{\rightarrow} Z_1\stackrel{\tau_0}{\rightarrow} Z_0=Z\rightarrow Y
  $$
  of local blowups such that each $Z_i\rightarrow Y\in e$ is nonsingular, and there exist morphisms $\alpha_i:Z_i\rightarrow Y_i$ for all $i$, giving a commutative diagram
  \begin{equation}\label{eqAb1}
  \begin{array}{lllllll}
   Y'=Y_r&\stackrel{\pi_{r-1}}{\rightarrow}&\cdots &\stackrel{\pi_1}{\rightarrow} &Y_1&\stackrel{\pi_0}{\rightarrow}& Y_0=Y\\  
  \uparrow\alpha_r&&&&\uparrow\alpha_1&&\uparrow\alpha_0=\lambda\\
   Z'=Z_r&\stackrel{\tau_{r-1}}{\rightarrow}& \cdots &\stackrel{\tau_1}{\rightarrow} &Z_1&\stackrel{\tau_0}{\rightarrow} &Z_0=Z
   \end{array}
   \end{equation}
   and further, the restriction $\nu_{\lambda\tau_0\cdots\tau_i}$ of $\nu_e$ to the field $K_{ \lambda\tau_0\cdots\tau_i}$ has value group $\Gamma_{\lambda\tau_0\cdots\tau_i}=\Gamma$ and   $A_{\lambda\tau_0\cdots\tau_i}$ has a regular system of parameters $z_1,\ldots,z_n$ (depending on $i$) such that $\nu(z_1),\ldots, \nu(z_n)$ is a free basis of $\Gamma$.
   
  We will construct the diagram (\ref{eqAb1}) by induction on $i$. Suppose that we have constructed $Z_j$  and morphisms $\tau_j$ and $\alpha_j$ for for $j\le i$. Let $R=\mathcal O_{Z_i,p_{Z_i}(e)}$.
  $R$ satisfies the assumptions of Lemma \ref{LemmaAbh2} (with $\omega=\nu_{\lambda\tau_0\cdots\tau_{i-1}}$)
  and $I=\mathcal I_{E_i,p_{Y_i}(e)}R$. We apply Lemma \ref{LemmaAbh2} to $R$ and $I$ to obtain $R\rightarrow S$ such that $IS$ is principal. Let $W$ be an open subset of $\alpha_i^{-1}(U_i)$ which contains $p_{Z_i}(e)$ and so that $z_1,\ldots,z_n$ are coordinates on $W$, and let $W_1$ be the blow up of $(M_1,\ldots,M_r)\mathcal O_W$ (with the notation of Lemma \ref{LemmaAbh2}).  There exists an open neighborhood $Z_{i+1}$ of $p_{W_1}(e)$ in $W_1$ such that $Z_{i+1}\rightarrow Y_0\in e$ is nonsingular and $\mathcal I_{E_i}\mathcal O_{Z_{i+1}}$ is locally principal. Thus there is a factorization $\alpha_{i+1}:Z_{i+1}\rightarrow Y_{i+1}$ by the universal property of the blow up. By our construction, we have that $A_{\lambda\tau_0\cdots\tau_{i+1}}=S^{\rm an}$ so that the value group $\Gamma_{\lambda\tau_0\cdots\tau_{i+1}}=\Gamma$ (by Lemma \ref{LemmaAbh1}).
  
 By induction, we may construct a diagram (\ref{eqAb1}) for any nonsingular $\pi\in e$. Since $A_{\pi}\subset A_{ \lambda\tau_0\cdots\tau_{r-1}}$ we have that the value group $\Gamma_{\pi}$ of the restriction of $\nu_e$ to the field $K_{\pi}$ is contained in $\Gamma$. Thus $\Gamma_e=\Gamma$ is isomorphic to $\ZZ^n$ as an unordered group.

  \section{Pathological behavior of the valuation associated to an \'etoile}
  
  Suppose that $K$ is an algebraic function field over a field $k$, and $\nu$ is a valuation of
  $K$ (which vanishes on $k\setminus \{0\}$). If $Z$ is a  proper model of  $K$ (the function field $k(Z)=K$), then there exists
  a unique (not necessarily closed) point $a\in Z$ such that the valuation ring $V_{\nu}$ in $K$ dominates the local ring 
  ${\mathcal O}_{Z,a}$. This point is called the center of $\nu$ on $Z$.
  
 Let $r$ be the rank of $\nu$, and let 
 $$
 0=P_0\subset P_1\subset \cdots \subset P_r
 $$
 be the chain of distinct prime ideals in $V_{\nu}$.

  Suppose that $Z_2$, $Z_1$ are proper models of $K$, and $Z_2$ dominates $Z_1$ in a neighborhood of the center of $\nu$. 
  Then we have a commutative diagram (in a neighborhood of the center of $\nu$)

  $$
  \begin{array}{rcl}
  \mbox{spec}(V)&\stackrel{\pi_2}{\rightarrow}& Z_2\\
  \pi_1 &\searrow &\downarrow\\
  &&Z_1
  \end{array}
  $$

  Let $W_j(i)$ be the Zariski closure of $\pi_i(P_j)$ in $Z_i$ for $0\le j\le r$.
  Then for all $j$,
  \begin{equation}\label{E1}
  \dim W_j(1)\le \dim W_j(2).
  \end{equation}
  
  In fact, after an appropriate blow up, the dimensions of the centers $W_j$ on $Z$ stablilize to $\mbox{trdeg}_k((V/P_j)_{P_j})$.
  
  The case of analytic spaces is completely different from that of algebraic varieties, as the inequality \ref{E1} does not hold for the centers of a valuation associated to an \'etoile.
  
  If $e\in \mathcal E_Y$ and $\pi:Y'\rightarrow Y\in e$, then there is a natural homomorphism
  $\mbox{Spec}(V_e)\rightarrow \mbox{Spec}({\mathcal O}_{Y',p_{\pi}(e)})$. Suppose that $Q$ is a prime ideal in $V_e$.  Let 
  $a'=Q\cap {\mathcal O}_{Y',p_{\pi}(e)}$, 
  a prime ideal in ${\mathcal O}_{Y',p_{\pi}(e)}$. In an analytic  neighborhood of $p_{\pi}(e)$ in $Y'$ we have an irreducible
  analytic set $Z(a')$. We will call this the center of $Q$ on $Y'$.
  
  \begin{Example}\label{pathval} There exists an \'etoile $e$ on $Y_0=\CC^4$ such that $V_e$ has rank larger than 1, and $V_e$ has 
   a proper prime ideal $Q$ such  that there exists an infinite chain
  $$
  \rightarrow \cdots \rightarrow Y_m\rightarrow \cdots \rightarrow Y_2\rightarrow Y_1\rightarrow Y_0
  $$
  with $Y_m\rightarrow Y_0\in e$ for all $m$, such that the center of $Q$ on $Y_m$ has dimension 3 if $m$ is even and the center of $Q$ on $Y_m$ has dimension 2 if $m$ is odd. 
  \end{Example}
  
  To construct the example, we need the following lemma.
  
  \begin{Lemma}\label{LemmaG3} Suppose that $K$ is a field and $R$ is a local  subring. Suppose that $q_1\subset q_2$ are distinct nonzero prime ideals in $R$.
  Then there exists a valuation ring $V$ of $K$ and nonzero prime ideals $p_1\subset p_2$ in $V$ such that $V$ dominates $R$, $p_1\cap R=q_1$ and
  $p_2\cap R=q_2$.
  \end{Lemma}
  
  \begin{proof} By Proposition 2.22 \cite{Ab2}, there exists a valuation ring $W$ of $K$ with maximal ideal $m_W$ such that $R\subset W$ and $m_W\cap R=p_1$.
  We have a commutative diagram where the horizontal arrows are inclusions,
  $$
  \begin{array}{lll}
  R&\rightarrow &W\\
  \downarrow&&\downarrow\\
  R/p_1&\rightarrow &W/m_W
  \end{array}
  $$
  
  Again by Proposition 2.22 \cite{Ab2}, there exists a valuation ring $D$ of $W/m_W$ such that $R/p_1\subset D$ and 
  $m_D\cap R/p_1=p_2/p_1$. Let $V=\pi^{-1}(D)$. $V$ is a valuation ring of $K$ as proven on page 57 of \cite{Ab2}. By Lemma 2.31 \cite{Ab2}, $Q=m_W\cap V$ satisfies $W=V_Q$. 
  
  By our construction, $V$ satisfies the conclusions of the lemma.
  \end{proof}
    
  We now recall an example in \cite{HLT}. Consider the local $\CC$-algebra homomorphism of analytic local rings 
  \begin{equation}\label{A1}
  \phi:\CC\{u,v,w\}\rightarrow \CC\{x,y\}
  \end{equation}
  defined by 
  $$
  u=y, v=ye^x, w=ye^{e^x}.
  $$
  We have that $\phi$ is 1-1. This can be seen as follows. Suppose that there is a nonzero series $\Lambda(u,v,w)$ such that 
  $\Lambda(y,ye^x,ye^{e^x})=0$. Collecting $y$ terms, we must then have that $1,e^x,e^{e^x}$ are algebraically dependent
  over $\CC$, which is a contradiction.
  
  Now make the substitution $u=u_1, v= u_1(v_1-1), w=u_1(w_1-e)$. We have that
  $u_1=y, v_1=e^x-1, w_1=e^{e^x}-e$. We thus have an induced local $\CC$-algebra homomorphism
  \begin{equation}\label{A2}
  \phi_1:\CC\{u_1,v_1,w_1\}\rightarrow \CC\{x,y\}.
  \end{equation}
  $\phi_1$ has a nontrivial  kernel which  is generated by  $w_1-e^{v_1+1}+e$.

  Now extend $\phi$ and $\phi_1$ to a commutative diagram of $\CC$-algebra homomorphisms 
  $$
  \begin{array}{lll}
  \CC\{t,u,v,w\}&\stackrel{\overline\phi}{\rightarrow}&\CC\{x,y\}\\
  \downarrow&\nearrow&\overline \phi_1\\
  \CC\{t_1,u_1,v_1,w_1\}&&
  \end{array}
  $$
  by defining $t=u_1t_1$ and $\overline \phi_1(t_1)=0$, so that $\overline \phi(t)=0$.
  Then the kernel of $\overline \phi$ is $(t)$ and the kernel of $(\overline \phi_1)$ is $(t_1,w_1-e^{v_1+1}+e)$.
  
  Let $R_0=\CC\{t,u,v,w\}$ with prime ideal $Q_0=(t)$ and $R_1=\CC\{t_1,u_1,v_1,w_1\}$ with prime ideal $Q_1=(t_1,w_1-e^{v_1+1}+e)$.
  Let $R_2=\CC\{t_2,u_2,v_2,w_2\}$ and define a $\CC$-algebra homomorphism $R_1\rightarrow R_2$ by the substitutions
  $$
  t_1=t_2, u_1=u_2, v_1=v_2, w_1-e^{v_1+1}+e=t_2w_2.
  $$
  Let $Q_2=(t_2)$. We have $Q_2\cap R_1=Q_1$ and $Q_1\cap R_0=Q_0$.
  
  Now we define $R_2\rightarrow R_3\rightarrow R_4$ to be the sequence $R_0\rightarrow R_1\rightarrow R_2$ with the variables
  $t_i, u_i,v_i,w_i$ for $0\le i\le 2$ changed to $t_{i+2},u_{i+2},v_{i+2},w_{i+2}$. Let $Q_i$ for $2\le i\le 4$ be the corresponding prime ideals in $R_i$. Repeating this construction we construct an infinite chain of convergent power series rings in four variables
  \begin{equation}\label{eqG1}
  R_0\rightarrow R_1\rightarrow \cdots
  \end{equation}
  such that the $R_i$ have prime ideals $Q_i$ such that
  $Q_i\cap R_{i-1}=Q_{i-1}$ for all $i$ and $Q_i$ has height 1 in $R_i$ if $i$ is even and $Q_i$ has height 2 in $R_i$ if $i$ is odd.
  
  Let $Y_0$ be the germ of $\CC^4$ at the origin $p_0$ which has local ring $\mathcal O_{Y_0,p_0}$. Let $Y_1$ be the blow up $p_0$, and let $p_1$ be the point of $Y_1$ whose local ring is $R_1$. Let $W$ be the germ of a nonsingular surface at $p_1$ which has local equations $t_1=w_1=0$, and let $Y_2$ be the blow up of $W$. Let $p_2$ be the point of $Y_2$ which has local ring $R_2$.
  Continuing this way, we see that the sequence of local rings (\ref{eqG1}) is a sequence of local rings of a sequence of local blowups
  \begin{equation}\label{eqG2}
  \cdots\rightarrow Y_1\rightarrow Y_0.
  \end{equation}

  Let $K_i$ be the quotient field of $R_i$ for $i\ge 0$. Set $L(0)=\lim_{\rightarrow}L_i$ and $A(0)=\lim_{\rightarrow} R_i$.
  There exists a valuation ring $V(0)$ of $L(0)$ which dominates $A(0)$, with prime ideals $Q(0)\subset m(0)$ such that
  $Q(0)\cap R_i=Q_i$ for all $i$ and $m(0)\cap R_i=m_{R_i}$ by Lemma \ref{LemmaG3}.
  
  Suppose that $W\rightarrow \mbox{Spec}(R_i)$ is a projective birational morphism such that $W$ is smooth. Let $p_W\in W$ be the unique point in the scheme $W$ such that $V(0)$ dominates $\mathcal O_{W,p_W}$. We further restrict to $W$ such that the center of 
  $Q(0)$ on $W$ is smooth at $p_W$.

  Let $\overline W$ be the germ of a complex analytic space associated to $W$ at $p_W$. The center of $Q(0)$ on $W$ is nonsingular, and extends uniquely to a prime ideal in $\mathcal O_{\overline W,p_W}$. Then the analytic local rings $\mathcal O_{\overline W,p_W}$ form a directed system as do their quotient fields. Let $A(1)$ be the limit of the local rings $\mathcal O_{\overline W,p_W}$, and let $L(1)$ be the quotient field of $A(1)$.

  Since $V(0)$ is the union of the local rings $\mathcal O_{W,p_W}$, we have that $A(1)$ dominates $V(0)$.
  Further, by our construction, there exists a prime ideal $Q'$ in $A(1)$ which dominates $Q(0)$.
  By Lemma \ref{LemmaG3} there exists a valuation ring $V(1)$ of $L(1)$ which dominates $A(1)$ (and thus dominates $V(0)$), with prime ideals $Q(1)\subset m(1)$ such that
  $Q(1)\cap A(1)=Q'$  and $m(1)\cap A(1)=m_{A(1)}$. We necessarily have that $V(1)\cap L(0)=V(0)$.
  
  We now construct a local ring $A(2)$ with a distinguished (non maximal) prime ideal $Q''$. Associated to any projective birational
   morphism $X\rightarrow \mbox{Spec}(\mathcal O_{\overline W,p_W})$ where $\overline W$ is a germ of an analytic space used in the 
  construction of $A(1)$, and $X$ is smooth with smooth center by $Q(1)$, we obtain an associated germ of a complex analytic space $\overline X$, and we have a directed system of local rings associated to such $\overline X$. Let $A(2)$ be the limit of these local rings, with quotient field $L(2)$.   That is,  $A(2)$ is the  union of all $\mathcal O_{X}^{\rm an}(U)$ with $U$ an open neighborhood in $X^{\rm an}$ of the center of $V(1)$. Again, we have that $A(2)$ is a local ring with a distinguished (nonmaximal) ideal $Q''$.
  We have that $V(1)\subset A(2)$ and $m_{A(2)}\cap V(1)=m_{V(1)}$, $Q''\cap V(1)=Q(1)$. 
  By Lemma \ref{LemmaG3} there exists a valuation ring $V(2)$ of $L(2)$ which dominates $A(2)$ (and thus dominates $V(1)$), with prime ideals $Q(2)\subset m(2)$ such that
  $Q(2)\cap A(2)=Q''$  and $m(2)\cap A(2)=m_{A(2)}$. We necessarily have that $V(2)\cap L(1)=V(1)$.
  
  We now repeat this construction over all natural numbers, starting  by applying the construction of $A(1)$ and then $A(2)$ from $A(0)$ to $A(2)$,   to construct an increasing sequence of
  fields $L(i)$ with valuation rings $V(i)$ (for $i\in \NN$) such that $V(i)$ contains a nonmaximal ideal $Q(i)$ with $L(i)\subset L(i+1)$ , $V(i+1)\cap L(i)=V(i)$, $Q(i+1)\cap V(i)=Q(i)$ and $m(i+1)\cap V(i)=m(i)$  for all $i$.
  
  Let $L=\lim L(i)$. $A=\lim V(i)$ is a local ring with a distinguished nonmaximal prime ideal $Q^*$. By Lemma \ref{LemmaG3} there exists a valuation ring $V$ of $L$ which dominates $A$ , with prime ideals $Q\subset m_V$ such that
  $Q\cap A=Q^*$  and $m_V\cap A=m_{A}$.

  Let $e_0$ be the subcategory of $\mathcal E(Y_0)$ of morphisms  of analytic spaces used in the construction of $V$ and $L$. Then $e_0$ satisfies 1) - 3) of Definition \ref{EtoileDef} of an \'etoile, so there exists by Zorn's lemma an \'etoile  $e\in \mathcal E_{Y_0}$ containing $e_0$ (Lemma 2.2 \cite{H}). 
  
  By Lemmas \ref{Lemma2} and \ref{Lemma3} $e$ is unique. In particular, we have that $K_e=L$ and $V_e=V$,
  so the conclusions of the example hold.


\begin{thebibliography}{1000000000}
\bibitem{Ab} S. Abhyankar, On the valuations centered in a local domain, Amer. J. Math. 78, 321 - 348 (1956).
\bibitem{Ab2} S. Abhyankar, Ramification Theoretic Methods in Algebraic Geometry, Princeton University Press, 1959.
\bibitem{Ab3} S. Abhyankar, Resolution of singularities of embedded algebraic surfaces, Academic Press, 1966.
\bibitem{AHV} J. M. Aroca, H. Hironaka and J. L. Vicente, Introduction to the theory of infinitely near singular points, The theory of maximal contact, Desingularization theorems, Memorias de matematica del Insituto ``Jorge Juan'' 28 (1974), 29 (1975), 30 (1977).

\bibitem{BS} C. Banica and O. Stanasila, Algebraic methods in the global theory of complex spaces, John Wiley and sons, 1976.
\bibitem{BM} E. Bierstone and P. Milman, Canonical desingularization in characteristic zero by blowing up the maxima strata of a local invariant, Invent. Math. 128 (1997), 207 - 302.
\bibitem{BM2} E. Bierstone and P. Milman Subanalytic Geometry, in Model Theory, Algebra and Geometry, MSRI Publications 39 (2000).

\bibitem{BM3} E. Bierstone and P. Milman, Semianalytic and subanalytic sets, Publ. Math. IHES 67 (1988), 5 -42.


\bibitem{C1} S.D. Cutkosky, Local monomialization and factorization of morphisms, Ast\'erisque 260, Soci\'et\'e math\'ematique de France, (1999).
\bibitem{C3} S.D. Cutkosky, Errata of Local monomialization and factorization of morphisms, http://faculty.missouri.edu/$\tilde{\,\,}$cutkoskys
\bibitem{C2} S.D. Cutkosky, Local monomialization of transcendental extensions, Ann, Inst. Fourier, 1517 - 1586, (2005).
\bibitem{C5} S.D. Cutkosky, Counterexamples to local monomialization in positive characteristic, Math. Ann. 362 (2015), 321 - 334.
\bibitem{C4} S.D. Cutkosky, Local monomialization of analytic maps, to appear in Advances in Math.
\bibitem{C6} S.D. Cutkosky, Rectilinearization of sub analytic sets as a consequence of local monomialization, arXiv:1601.02482.
\bibitem{De} J. Denef, Monomialization of morphisms and $p$-adic quantifier elimination, Proc. Amer. Math. Soc. 141 (2013), 2569 - 2574.
\bibitem{DD} J. Denef and L. van den Dries, $p$-adic and real sub analytic sets, Annals of Mathematics 128 (1988), 79-138.
\bibitem{D} A. Douady, Le probl\`eme des modules pour les sous-espaces analytiques compacts d'un espaces analytique donn\'e,
Ann. Inst. Fourier (1966), 1-95.
\bibitem{Gab} A.M. Gabri\`elov, Formal relations between analytic functions, Math USSR IZV. 7 (1973) 1056 - 1088.
\bibitem{Gab2} A.M. Gabri\`elov, The formal relations between analytic varieties, Functional Anal. i Prilozen (1973), 18 - 32.

\bibitem{EGAIV} A. Grothendieck, and A. Dieudonn\'e, El\'ements de g\'eom\'etrie alg\'ebrique IV, vol. 2, Publ. Math. IHES 24 (1965), vol. 4, Publ. Math. IHES 32 (1967).
\bibitem{Ha} R. Hartshorne, Algebraic Geometry, Springer Verlag (1977).
\bibitem{H1} H. Hironaka, Resolution of Singularities of an algebraic variety over a field of characteristic zero,  Annals of Math., 79 (1964), 109-326.
\bibitem{H2} H. Hironaka, Desingularization of complex-analytic varieties, in Actes du Congr\`es International des mathe\'ematicains (Nice, 1970) Tome 2, 627 - 631, Gauthier-Cillars, Paris, 1971.
\bibitem{H3} H. Hironaka, Introduction to real-analytic sets and real-analytic maps, Quaderni dei Gruppoi di Recerca Matematica del Consiglio Nazionale delle Ricerche, Insitutot Matematicao ``L. Tonelli'' dell'Universit\`a di Pisa, Pisa, 1973.
\bibitem{H} H. Hironaka, La Voute \'Etoil\'ee, in Singularit\'es \`a Carg\`ese, Ast\'erisque 7 and 8, (1973).
\bibitem{HLT} H. Hironaka, M. Lejeune-Jalabert and  B. Teissier, Platificateur local en geometrie analytique et aplatissement local,
in Singularit\'es \`a Carg\`ese, Ast\'erisque 7 and 8, (1973).
\bibitem{Li} B. Lichtin, Uniform bounds for some exponential sums (mod $p^r$) in two variables. Proceedings of the session in Analytic Number Theory an Diophantine Equations,
 63 pp. Bonner Math Schriften, 360, Univ. Bonn, Bonn, 2003.
\bibitem{L} S. Lojasiewicz, Introduction to complex analytic geometry, Birkhauser, (1991).
\bibitem{Ma} H. Matsumura, Commutative Algebra, 2nd edition, Benjamin/Cummings (1980).
\bibitem{N} M. Nagata, Local Rings, Wiley Interscience (1962).
\bibitem{Z} O. Zariski, The compactness of the Riemann manifold of an abstract field of algebraic functions, Bull. Amer. Math. Soc. 45 (1944), 683 - 691.
\bibitem{ZS} O. Zariski and P. Samuel, Commutative Algebra Vol II, Van Nostrand (1960).
\bibitem{LU} O. Zariski, Local Uniformization of algebraic varieties, Annals of Math. 41 (1940), 852 - 896.

\end{thebibliography}
\end{document}